\documentclass[11pt,reqno]{amsart}
\usepackage{fullpage}
\usepackage{amsmath,amsthm,amssymb,mathscinet,mathtools,mathrsfs}
\usepackage{graphicx,xcolor,pgfplots}
\usepackage{verbatim}
\pgfplotsset{compat=1.17}
\usepackage{subcaption}
\usepackage[backref=section]{hyperref}
\hypersetup{
	colorlinks=true,
	linkcolor={red!60!black},
	citecolor={green!60!black},
	urlcolor={blue!60!black},
}

\usepackage[left=1.25in, right=1.25in, top=1.2in, bottom=1.25in]{geometry}

\usepackage{enumitem}
\newcommand{\customlabel}[2]{#2\def\@currentlabel{#2}\label{#1}}

\def\alabel{\upshape({\alph*})}

\def\ilabel{\upshape({\roman*})}

\usepackage{scalerel}
\usepackage{dsfont}
\usepackage{blindtext}
\usepackage{multicol}
\newtheorem{theorem}{Theorem}[section]

\newtheorem{lemma}[theorem]{Lemma}

\newtheorem{proposition}[theorem]{Proposition}
\newtheorem{corollary}[theorem]{Corollary}

\newtheorem{observation}[theorem]{Observation}

\newtheorem{problem}[theorem]{Problem}

\numberwithin{equation}{section}

\def\eps{\varepsilon}
\newcommand{\ep}{\varepsilon}

\newcommand{\floor}[1]{\lfloor#1\rfloor}
\newcommand{\ceil}[1]{\lceil#1\rceil}
\newcommand{\bigfloor}[1]{\left\lfloor#1\right\rfloor}
\newcommand{\bigceil}[1]{\left\lceil#1\right\rceil}

\newcommand{\tbf}[1]{\textbf{#1}}

\makeatletter
\newcommand*{\rom}[1]{\expandafter\@slowromancap\romannumeral #1@}
\makeatother

\def\COMMENT#1{}
\let\COMMENT=\footnote

\usepackage{caption}

\tikzstyle{every node}=[circle, draw, fill=black!50, inner sep=0pt, minimum width=2pt]

\usetikzlibrary{patterns}
\usetikzlibrary{decorations.markings, decorations.pathreplacing}

\tikzset{middlearrow/.style={
  decoration={markings,
   mark= at position 0.5 with {\arrow[very thick]{#1}} ,
  },
  postaction={decorate}
 }
}

\usepackage{etoolbox}

\allowdisplaybreaks

\title[Arbitrary orientations of Hamilton cycles in directed graphs of large minimum degree]{Arbitrary orientations of Hamilton cycles in directed graphs of large minimum degree}

\author{Louis DeBiasio \and Andrew Treglown}

\thanks{
\\ \indent LD: Department of Mathematics, Miami University, Oxford, OH. Email: \texttt{debiasld@miamioh.edu}. Research supported in part by NSF grant DMS-1954170.
\\ \indent AT: School of Mathematics, University of Birmingham, United Kingdom. Email: \texttt{a.c.treglown@bham.ac.uk}. Research supported by EPSRC grants EP/V002279/1 and EP/V048287/1.}

\begin{document}

\begin{abstract}
In 1960, Ghouila-Houri proved that every strongly connected directed graph $G$ on $n$ vertices with minimum degree at least $n$ contains a directed Hamilton cycle.  We asymptotically generalize this result by proving the following: every 
directed graph $G$ on $n$ vertices with minimum degree at least $(1+o(1))n$ contains every orientation of a Hamilton cycle,  except for the directed Hamilton cycle in the case when $G$ is not strongly connected. In fact, this minimum degree condition forces every orientation of a cycle in $G$ of every possible length, other than perhaps the directed cycles. 
\end{abstract}

\maketitle

\section{Introduction}
\subsection{Hamilton cycles in directed graphs}
In this paper we give an asymptotic generalization of one of the cornerstone results in the study of directed graphs, Ghouila-Houri's theorem~\cite{GH}. 
Throughout, the \emph{digraphs} we consider do not have loops and we allow at most one edge in each
direction between any pair of vertices. The \emph{minimum  degree $\delta (G)$} of a digraph $G$ is the minimum number of edges incident to a vertex in $G$. A \emph{directed Hamilton cycle} in an $n$-vertex digraph is a cycle $(v_1, v_2, \dots , v_n, v_{n+1} = v_1)$ in which every edge $v_iv_{i+1}$ ($i \in [n]$) is oriented from $v_i$ to $v_{i+1}$. The notion of a   \emph{directed Hamilton path} is defined analogously.
\begin{theorem}[Ghouila-Houri~\cite{GH}]\label{ghthm}
If $G$ is a strongly connected digraph on $n \geq 2$ vertices with  $\delta (G)\geq n$, then $G$ contains a directed Hamilton cycle.
\end{theorem}
Ghouila-Houri~\cite{GH} also determined the minimum degree threshold for forcing a directed Hamilton path.
\begin{corollary}[Ghouila-Houri~\cite{GH}]\label{ghcor}
If $G$ is digraph on $n \geq 2$ vertices with  $\delta (G)\geq n-1$, then $G$ contains a directed Hamilton path.
\end{corollary}

Note that the bounds on $\delta (G)$ in Theorem~\ref{ghthm} and Corollary~\ref{ghcor} are tight as witnessed by the usual examples of a complete bipartite digraph with parts of sizes $\ceil{\frac{n}{2}}-1$ and $\floor{\frac{n}{2}}+1$, and the disjoint union of two complete digraphs of sizes $\lfloor \frac{n}{2} \rfloor$ and $\lceil \frac{n}{2} \rceil$, respectively.  Moreover, one cannot 
omit the strong connectivity condition in Theorem~\ref{ghthm} as there 
are $n$-vertex digraphs $G$ with $\delta (G) = \floor{\frac{3n}{2}} - 2$
which are not strongly connected and thus do not contain a directed Hamilton cycle.

As discussed further below, there has been significant interest in considering other orientations of Hamilton cycles in digraphs.
The main result of this paper is the following asymptotic generalization of Theorem~\ref{ghthm}.
\begin{theorem}\label{thm:main}
For all $\eta>0$, there exists $n_0\in \mathbb N$ such that if $G$ is a digraph on $n\geq n_0$ vertices with $\delta(G)\geq (1+\eta)n$, then $G$ contains every orientation of a Hamilton cycle,  except for the directed Hamilton cycle in the case when $G$ is not strongly connected.
\end{theorem}
 Theorem~\ref{thm:main} yields the following simple consequence. 
\begin{corollary}\label{cor1}
For all $\eta>0$, there exists $n_0\in \mathbb N$ such that if $G$ is a digraph on $n\geq n_0$ vertices with $\delta(G)\geq (1+\eta)n$, then $G$ contains every orientation of a Hamilton path.
\end{corollary}
 
Note that the minimum  degree conditions in Theorem~\ref{thm:main} and Corollary~\ref{cor1} are asymptotically tight as the examples given after Corollary~\ref{ghcor} do not contain a Hamilton cycle (resp.~path) of any orientation.

A crucial tool in our proof of Theorem~\ref{thm:main} is a result (Proposition~\ref{prop:structure}) that provides significant structural information about any $n$-vertex digraph $G$ with $\delta (G) >(1+o(1))n$. We believe that this result will likely have further applications.

\smallskip

The study of Hamiltonicity in digraphs has a rich history and, as we will now discuss, Theorem~\ref{thm:main} can be viewed as an asymptotic generalization of a couple of research strands in this area. For a wider survey on Hamilton cycles in digraphs, see~\cite{survey}.

Hamilton cycles and paths have been well-studied in \emph{tournaments} (i.e., orientations of complete graphs). A simple argument of R\'edei~\cite{redei} shows that every tournament 
contains a directed Hamilton path.
Motivated by this, in 1971
Gr\"unbaum~\cite{G71} proved that, other than three small exceptional tournaments, every tournament contains an \emph{anti-directed} Hamilton path.\footnote{A digraph is \emph{anti-directed} if it does not contain a directed path on $2$ edges.}
 Rosenfeld~\cite{R72}  provided a short proof of Gr\"unbaum's theorem, and also conjectured that every sufficiently large tournament contains every orientation of a Hamilton path. More than a decade later,
Thomason~\cite{T86} resolved Rosenfeld's conjecture. 
Finally, Havet and  Thomass\'e~\cite{HavT} strengthened this result,
 proving that, other than Gr\"unbaum's~\cite{G71} three exceptional tournaments, every tournament contains every orientation of a Hamilton path.

A classical result of Camion~\cite{camion} states that a tournament contains a directed Hamilton cycle if and only if it is
strongly connected.
Thomassen~\cite{thomassen} (for even $n \geq 50$) and
 Rosenfeld~\cite{R74} (for even $n \geq 28$) proved that every $n$-vertex tournament contains an anti-directed Hamilton cycle. The latter also conjectured that every sufficiently large tournament contains every orientation of a Hamilton cycle (except perhaps the directed Hamilton cycle). After various partial results, this conjecture was proven for all tournaments on at least $68$ vertices by Havet~\cite{Hav}.

These results therefore provide a  natural class of $n$-vertex digraphs $G$ with $\delta(G)=n-1$ that contain every orientation of a Hamilton path and Hamilton cycle (except perhaps the directed Hamilton cycle). On the other hand, Theorem~\ref{thm:main} and Corollary~\ref{cor1} show that  analogues of these statements hold for \emph{all} sufficiently large digraphs at the expense of a slightly larger minimum degree condition. 
It would be interesting to obtain  
results that unify these two settings.
For example, it may be the case that every sufficiently large $n$-vertex digraph $G$ (other than a class of exceptional cases) with $\delta(G) \geq n-1$ contains every orientation of a Hamilton path.
This and related questions are discussed  in Section~\ref{sec:conc}.

Seeking to generalize  Rosenfeld's~\cite{R74} and Thomassen's~\cite{thomassen} work on tournaments, in 1980 Grant~\cite{G80} 
asked whether every  even order $n$-vertex digraph $G$ 
with $\delta (G)\geq n-1$ contains an anti-directed Hamilton cycle. However, Cai~\cite{Cai} produced an example of an $n$-vertex digraph $G$ 
with $\delta (G)= n$ that does not contain an anti-directed Hamilton cycle, for every even $n\in \mathbb N$. This led the first author and Molla~\cite[Conjecture 1.7]{DM} to refine the conjecture Grant~\cite{G80} hinted at: every 
even order $n$-vertex digraph $G$ 
with $\delta (G)\geq n+1$ contains an anti-directed Hamilton cycle. Thus, Theorem~\ref{thm:main} asymptotically resolves this conjecture.

There has also been  interest in the minimum semi-degree threshold for forcing a given orientation of a Hamilton cycle in a digraph. Given a digraph $G$, its \emph{minimum semi-degree $\delta^0(G)$} is the minimum of all the in- and outdegrees of the vertices in $G$. Ghoulia-Houri's theorem implies that every $n$-vertex digraph $G$ with $\delta^0(G)\geq n/2$ contains a directed Hamilton cycle; indeed, such a digraph must be strongly connected. Moreover, the bound on $\delta^0(G)$ is sharp.
Grant~\cite{G80} gave a minimum semi-degree condition for forcing an anti-directed Hamilton cycle. After a few further partial results~\cite{BJMPT, HT_di, PT09}, the first author and Molla~\cite{DM} proved the following: every sufficiently large even order $n$-vertex digraph $G$ with $\delta^0(G)\geq n/2+1$ contains an anti-directed Hamilton cycle; in fact, provided $G$ is not one of two extremal examples, a minimum semi-degree of $\delta^0(G)\geq n/2$ suffices here.

Moving beyond anti-directed Hamilton cycles, we now have a complete picture of the minimum semi-degree threshold for forcing an arbitrary orientation of a Hamilton cycle. In 1995, H{\"a}ggkvist and Thomason~\cite{HT_di} proved that every sufficiently large 
$n$-vertex digraph $G$ with $\delta^0(G)\geq n/2+n^{5/6}$ contains \emph{every} orientation of a Hamilton cycle.  Twenty years later, DeBiasio, K\"uhn, Molla, Osthus and Taylor~\cite{DKMOT} provided the following sharp version of this result.
\begin{theorem}[DeBiasio, K\"uhn, Molla, Osthus and Taylor~\cite{DKMOT}]\label{minsemithm}
There exists $n_0 \in \mathbb N$ such that for all $n\geq n_0$, if $G$ is an $n$-vertex digraph with $\delta^0(G)\geq n/2$, then $G$ contains every orientation of a Hamilton cycle that is not anti-directed.  
\end{theorem}
Note that since $\delta (G) \geq 2\delta^0(G)$, Theorem~\ref{thm:main} asymptotically generalizes Theorem~\ref{minsemithm}.
It would be interesting (though likely very challenging) to prove a sharp version of Theorem~\ref{thm:main} \`a la Theorem~\ref{minsemithm}; this is discussed further in Section~\ref{sec:conc}.

\subsection{Pancylicity}
Our work also has implications to cycles of shorter length.
 An \emph{oriented graph} is a digraph with at most one edge between any pair of vertices.
An \emph{oriented cycle} (resp.\ \emph{oriented path}) is an oriented graph whose underlying graph is a cycle (resp.~path). 
A \emph{directed cycle} (resp.~\emph{directed path}) is an oriented cycle (resp.~path) in which all edges are oriented in the same direction.
The first part of the following result
is a straightforward consequence of Theorem~\ref{thm:main}, a lemma of Taylor~\cite{Tay}, and a couple of our auxiliary results; the moreover part of the theorem is  a simple consequence of a generalization~\cite{aldred} of Bondy's  pancyclicity theorem~\cite{bondy} (and does not use, e.g., the lemma of Taylor~\cite{Tay}).
We defer its proof to the appendix.
\begin{theorem}\label{cor2}
   For all $k \in \mathbb N$ and $\gamma >0$, there exists $n_0\in \mathbb N$ such that the following holds. If $G$ is a digraph on $n\geq n_0$ vertices with $\delta(G)\geq (1+\frac{1}{k+1}+\gamma )n$, then $G$ contains every oriented cycle on at most $n$ vertices, except for perhaps the directed  cycles on more than $\lceil \frac{n}{k} \rceil$ vertices.
   Moreover, if $G$ is a digraph on $n \geq 2$ vertices with  $\delta(G)\geq \floor{\frac{3n}{2}}-1$, then $G$ contains every oriented cycle of every possible length.
\end{theorem}

Theorem~\ref{cor2} immediately implies the following pancyclicity-type extension of Theorem~\ref{thm:main}.
\begin{corollary}\label{corlab}
  For all $\eta>0$, there exists $n_0\in \mathbb N$ such that if $G$ is a digraph on $n\geq n_0$ vertices with $\delta(G)\geq (1+\eta)n$, then $G$ contains every oriented cycle on at most $n$ vertices, except perhaps the directed cycles.
\end{corollary}

Note that  Theorem~\ref{cor2}, including the moreover part, is  (asymptotically) best-possible in a  strong sense. Indeed, let $ k, n \in \mathbb N$ with $n\geq k+1$. Consider the $n$-vertex digraph $G_1$ obtained from the transitive tournament $T$ on $k+1$  vertices by blowing-up and replacing each vertex of $T$ with a complete digraph on $\ceil{\frac{n}{k+1}}$  or $\floor{\frac{n}{k+1}}$ vertices. Now we have $\delta(G_1)= n+ \floor{\frac{n}{k+1}}-2 $ and $G_1$ contains no directed cycle of length more than $ \ceil{\frac{n}{k+1}}$.
In particular, having $\delta(G)$ almost $(1+\frac{1}{k+1})n$ does not guarantee directed cycles of length more than $ \ceil{\frac{n}{k+1}}$; however, as soon as $\delta(G)$ is somewhat bigger than $(1+\frac{1}{k+1})n$, Theorem~\ref{cor2} tells us that $G$ in fact contains every directed cycle of length up to $ \ceil{\frac{n}{k}}$.
Thus, Theorem~\ref{cor2} asymptotically determines the minimum degree threshold for forcing a directed cycle of a given length in an $n$-vertex digraph:
\begin{corollary}
    For any $k \in \mathbb N$ and any $\gamma >0$, there exists $n_0\in \mathbb N$ such that the following holds. 
    Suppose that $n \geq n_0$ and let $C$ be a directed cycle of length $\ceil{\frac{n}{k+1}}<|C|\leq \ceil{\frac{n}{k}}$.
    If $G$ is a digraph on $n\geq n_0$ vertices with $\delta(G)\geq (1+\frac{1}{k+1}+\gamma )n$, then $G$ contains $C$. On the other hand, for every $n\geq k+1$  there is an $n$-vertex digraph $G_1$ with $\delta(G_1)= n+ \floor{\frac{n}{k+1}}-2 $ that does not contain $C$.
\end{corollary}

\subsection{Directed $2$-factors}

 Recall that the moreover part of Theorem~\ref{cor2} implies that, without the assumption of strong connectivity, the minimum
degree threshold for forcing 
 a directed Hamilton cycle in an $n$-vertex digraph $G$ is $n+\floor{\frac{n}{2}}-1$.
 The following observation generalizes this fact.  
 Note that a \emph{directed $2$-factor} in a digraph $G$ is a collection of vertex-disjoint directed cycles that together cover all vertices in $G$.

\begin{observation}\label{obs:2factor}
Given any $k\in \mathbb N$, if $G$ is a digraph on $n\geq 2(k+1)$ vertices with $\delta(G)\geq n+\floor{\frac{n}{k+1}}-1$, then every strongly connected component $H$ of $G$ has more than $\floor{\frac{n}{k+1}}\geq 2$ vertices and satisfies $\delta(H)\geq |V(H)|+\floor{\frac{n}{k+1}}-1$. Consequently, Theorem~\ref{ghthm} implies that $G$ contains a directed 2-factor with at most $k$ cycles.   
\end{observation}

Note that the minimum degree condition in Observation~\ref{obs:2factor} is tight as witnessed by the digraph $G_1$ defined after Corollary~\ref{corlab}.


\subsection*{Organization of the paper}
In Section~\ref{sec:structure} we state and prove one of our key auxiliary results, Proposition~\ref{prop:structure}, which  essentially is a `robust' version of Observation~\ref{obs:2factor}. 
In Section~\ref{sec:embed} we then show that the structure obtained from Proposition~\ref{prop:structure} allows one to prove Theorem~\ref{thm:main}. In Section~\ref{sec:conc}, we end the paper with a few concluding remarks and open questions.
 Note that we defer the proof Theorem~\ref{cor2}, as well as Theorem~\ref{thm:link}, to the appendix.

\subsection*{Notation}
Throughout, $\mathbb N$ denotes the set of positive integers (i.e., it does not contain $0$).

Let $G$ be a digraph. We define $|G|:=|V(G)|$ and $e(G):=|E(G)|$. Given $x \in V(G)$, we write $N^+_G(x)$ for the \emph{out-neighborhood of $x$ in $G$} and write $N^-_G(x)$ for the \emph{in-neighborhood of $x$ in $G$}.
Thus, $|N^+_G(x)|=d^+_G(x)$ and $|N^-_G(x)|=d^-_G(x)$. Given $Y\subseteq V(G)$ we define
$N^+_G(x,Y):= N^+_G(x) \cap Y$ and 
$N^-_G(x,Y):= N^-_G(x) \cap Y$.
Set $d^+_G(x,Y):=|N^+_G(x,Y)|$ and 
$d^-_G(x,Y):=|N^-_G(x,Y)|$, and let 
$d_G(x,Y):=d^+_G(x,Y)+ d^-_G(x,Y)$. We define $d_G(x):=d_G(x,V(G))$.

Given distinct $x,y \in V(G)$, we write $xy$  for the edge directed from $x$ to $y$ in $G$; so $xy$ and $yx$ are different edges. We say that $x$ sends a \emph{double edge} to $y$ in $G$ if $xy, yx \in E(G)$.

Given subsets $A,B\subseteq V(G)$ (not-necessarily disjoint), let $E^+_G(A,B)$ 
be the set of all $xy\in E(G)$ such that $x\in A$ and $y\in B$. 
We set $E^-_G(A,B):= E^+_G(B,A)$ and 
let $E_G(X,Y):=E^+ _G(X,Y)\cup E^-_G(X,Y)$.  
Let $e^+_G(X,Y):=|E^+_G(X,Y)|$, $e^-_G(X,Y):=|E^- _G(X,Y)|$, and $e_G(X,Y)=|E_G(X,Y)|$.


Given $X \subseteq V(G)$, we write $G[X]$ for the subdigraph of $G$ induced by $X$. We write $G\setminus X$ for the  subdigraph of $G$ induced by $V(G)\setminus X$.

Given a digraph $G$, a \textit{sink} is a vertex of out-degree $0$ and a \textit{source} is a vertex of in-degree $0$.  
Given an oriented cycle $C$, a vertex $v\in V(C)$ is a \emph{switch} if it is either a source or a sink in $C$. 
A digraph is \emph{strongly connected} if, for every ordered pair of distinct vertices $x,y \in V(G)$, there is a directed path from $x$ to $y$ in $G$.

When we state that a digraph $G$ contains \emph{every orientation of a Hamilton cycle}, we mean it contains every oriented cycle on $|G|$ vertices.

Given a graph $G$, we write $e(G)$ for the number of edges in $G$. Given $x \in V(G)$ we write $d_G(x)$ for the \textit{degree of $x$ in $G$}. Given two  graphs $G_1$ and $G_2$, we write $G_1+G_2$ for the \emph{join of $G_1$ and $G_2$}; that is, 
$G_1+G_2$ is obtained from the disjoint union of $G_1$ and $G_2$ by adding all possible edges between $G_1$ and $G_2$.

Throughout the paper, we
omit all floor and ceiling signs whenever these are not crucial.
The constants in the hierarchies used to state our results are chosen from right to left.
For example, if we claim that a result holds whenever $0< a\ll b\ll c\le 1$, then 
there are non-decreasing functions $f:(0,1]\to (0,1]$ and $g:(0,1]\to (0,1]$ such that the result holds
for all $0<a,b,c\le 1$  with $b\le f(c)$ and $a\le g(b)$.  
Note that $a \ll b$ implies that we may assume in the proof that, e.g., $a < b$ or $a < b^2$.

\section{Partitioning into robust expanders}\label{sec:structure}

In this section we will make use of the concept of \emph{robust expansion}, a notion introduced by K\"uhn, Osthus and Treglown~\cite{robust} that has found applications to a range of graph embedding problems.

Let $0<\nu\leq \tau <1/2$ and let $G$ be a digraph on $n$ vertices.  For $S\subseteq V(G)$, define $RN^+_\nu(S):=\{v \in V(G): d_G^-(v, S)\geq \nu n\}$ to be the \emph{$\nu$-robust out-neighborhood} of $S$.  We say that $G$ is a \emph{robust $(\nu, \tau)$-{outexpander}} if $|RN^+_\nu(S)|\geq |S|+\nu n$ for all $S\subseteq V(G)$ with $\tau n\leq |S|\leq (1-\tau)n$. 

The following structural result can be viewed as a robust version of Observation~\ref{obs:2factor}.  
\begin{proposition}\label{prop:structure}
Let $k\in \mathbb N$ and let $0<1/n_0 \ll \nu\ll\tau\ll\alpha\ll\zeta \ll 1/k$. If $G$ is a digraph on $n\geq n_0$ vertices with $\delta(G)\geq (1+\frac{1}{k+1}+\zeta)n$, then there exists a partition $\{V_1, \dots, V_t\}$ of $V(G)$ such that:
\begin{enumerate}
\item for all $i\in [t]$, $|V_i|\geq (\frac{1}{k+1}+\frac{\zeta}{2})n$, and consequently $t\leq k$;
\item for all $i\in [t]$, $G[V_i]$ is a robust $(\nu,\tau)$-outexpander with $\delta(G[V_i])\geq (1+\frac{1}{k+1}+\frac{\zeta}{2})|V_i|$;
\item if $t\geq 2$, then for all $1\leq i<j\leq t$, $e^-_G(V_i, V_j)> \frac{n^2}{(k+1)^2}$; 
\item if $t\geq 2$, then $\displaystyle \sum_{1\leq i<j\leq t}e^+_G(V_i, V_j)\leq \alpha \sum_{1\leq i<j\leq t}|V_i||V_j|\leq \alpha n^2$.
\end{enumerate}
\end{proposition}
Note that Proposition~\ref{prop:structure} implies that an $n$-vertex digraph $G$ with $\delta(G)>(1+\frac{1}{k+1}+o(1))n$ is either a robust outexpander  (if $t=1$) or has structure resembling that of the example described after Corollary~\ref{corlab}. More precisely, if $t\geq 2$ in Proposition~\ref{prop:structure},  the vertex classes $V_i$  are of reasonable size and induce robust outexpanders of large minimum degree (rather than complete digraphs). Further, the edges between vertex classes are  oriented approximately as in the blow-up of a transitive tournament. 

Roughly speaking, K\"uhn, Osthus and Treglown~\cite{robust}
proved that if  a  sufficiently large digraph $G$ is a robust outexpander with linear minimum semi-degree, then $G$ contains a directed Hamilton cycle. Moreover, Taylor~\cite{Tay} proved that such a digraph $G$ actually contains every orientation of a Hamilton cycle (see Theorem~\ref{thm:ham} in Section~\ref{sec:embed}).

This latter result is thus very useful when combined with Proposition~\ref{prop:structure}. Indeed, given a digraph as in Theorem~\ref{thm:main}, one may apply Proposition~\ref{prop:structure} (with $k\in \mathbb N$ chosen so that $1/(k+1) \leq \eta /2$ say) to obtain the partition $\{V_1,\dots, V_t\}$ of $V(G)$. Given an arbitrary orientation $C$ of an $n$-vertex cycle, Taylor's theorem allows one to find a segment of $C$ in each $G[V_i]$ covering all of $V_i$. The challenge therefore is to combine these segments into a copy of $C$ in $G$. This will be achieved in Section~\ref{sec:embed}.

Given a digraph $G$ and a partition $\{X_1, X_2\}$ of $V(G)$, we say that $(X_1, X_2)$ is an \emph{$\alpha$-sparse cut} if $e^+_G(X_1, X_2)\leq \alpha |X_1||X_2|$ and $X_1, X_2$ are non-empty.   
The following lemma says that if $G$ is a digraph on $n$ vertices with $\delta(G)> (1+o(1))n$, then there is only one way for $G$ to fail to be a robust outexpander.  

\begin{lemma}\label{lem:sparse}
Let $0<\nu, \tau, \alpha, \eta<1$ such that $\tau<{1}/{2}$ and $\nu\leq {\alpha \tau \eta}/{4}$, and let $G$ be a digraph on $n$ vertices with $\delta(G)\geq (1+\eta)n$.  If $G$ has no $\alpha$-sparse cut, then $G$ is a robust $(\nu,\tau)$-outexpander. 
\end{lemma}

\begin{proof}
Suppose $G$ has no $\alpha$-sparse cut, and suppose for a contradiction  there exists $S\subseteq V(G)$ with $\tau n\leq |S|\leq (1-\tau)n$ such that $|RN_\nu^+(S)|<|S|+\nu n$.  Let $A:=S\setminus RN_\nu^+(S)$, $B:=S\cap RN_\nu^+(S)$, $C:=RN_\nu^+(S)\setminus S$, and $D:=V(G)\setminus (A\cup B\cup C)$.  

First note that we have 
\begin{equation}\label{eq:CA}
(\alpha\tau-\nu) n<|C|<|A|+\nu n
\end{equation}
where the upper bound holds by the assumption on $S$; the lower bound holds because if $|C|\leq (\alpha\tau-\nu)n$, then we would have
\begin{align*}
e^+_G(S, C\cup D)\leq |S||C|+\nu n(n-|S|)&=\left(\frac{|C|}{n-|S|}+\frac{\nu n}{|S|}\right)|S|(n-|S|)\\
&\leq \left(\frac{\alpha\tau-\nu}{\tau}+\frac{\nu}{\tau}\right)|S|(n-|S|)= \alpha |S|(n-|S|), 
\end{align*}
a contradiction to the fact that there is no $\alpha$-sparse cut.

We will  obtain a contradiction by counting the sum of the degrees of the vertices in $A$ in two different ways.
  We have 
\begin{align*}
|A|(1+\eta)n & \leq \sum_{v\in A}d_G(v)\\
&=\sum_{v\in A}d^+_G(v,A)+\sum_{v\in A}d^-_G(v, A\cup B)+e^+_G(A, B)+e_G(A,C)+e_G(A,D)\\
&=\sum_{v\in A}d^-_G(v,A)+\sum_{v\in A}d^-_G(v, A\cup B)+e^+_G(A, B)+e_G(A,C)+e_G(A,D)\\ 
&< 2\nu n|A|+|A||B|+2|A||C|+|D|(|A|+\nu n)\\
&=|A|(|B|+2|C|+|D|+2\nu n)+|D|\nu n\\
&\stackrel{\eqref{eq:CA}}{<} |A|(|A|+|B|+|C|+|D|+3\nu n)+|D|\nu n\\
&= |A|(1+3\nu)n+|D|\nu n.
\end{align*}
Thus, $(\eta -3 \nu)|A|< \nu |D|$
which, since $2 \nu \leq \alpha \tau /2$ and $3\nu \leq \eta /2$ implies that
$$\alpha \tau \eta n/4 \leq (\alpha\tau-2\nu)(\eta-3\nu)n\stackrel{\eqref{eq:CA}}{<} (\eta-3\nu)|A|< \nu|D|\leq \nu n,$$ a contradiction.
\end{proof}

The next technical lemma allows us to ``clean-up'' a sparse cut to ensure it has some additional properties (which will be needed to prove Proposition~\ref{prop:structure}).

\begin{lemma}\label{lem:clean}
Let $k\in \mathbb N$, let $0<\zeta<1-\frac{1}{k+1}$, and let $0<\alpha<\frac{\zeta}{24(k+1)}$.  Let $G$ be a digraph on $n$ vertices with $\delta(G)\geq (1+\frac{1}{k+1}+\zeta)n$ and let $X\subseteq V(G)$ such that 
\begin{enumerate}[label=\alabel]
\item\label{it:a} $|X|> (\frac{1}{k+1}+\frac{\zeta}{2})n$;
\item\label{it:b} $\delta(G[X])\geq (1+\frac{1}{k+1}+\zeta-\alpha)|X|$;
\item\label{it:c} $d_G(v,X)\geq |X|+(\frac{1}{k+1}+\zeta-\alpha)n$ for all but at most $\alpha^2n$ vertices $v\in X$. 
\end{enumerate}

If $G[X]$ has an $\alpha^2$-sparse cut, then $G[X]$ has an $\alpha$-sparse cut $(V_1, V_2)$ such that for all $i\in [2]$:
\begin{enumerate}[label=\ilabel]
\item\label{it:i} $|V_i|> (\frac{1}{k+1}+\frac{\zeta}{2})n$;
\item\label{it:ii} $\delta(G[V_i])\geq (1+\frac{1}{k+1}+\zeta-10(k+1)\alpha)|V_i|$;
\item\label{it:iii} $d_G(v,V_i)\geq |V_i|+(\frac{1}{k+1}+\zeta-10(k+1)\alpha)n$ for all but at most $\alpha n$ vertices $v\in V_i$.
\end{enumerate}
\end{lemma}

Note that \ref{it:i} automatically implies that if $X$ is a set meeting the requirements \ref{it:a}, \ref{it:b}, \ref{it:c} and $(\frac{1}{k+1}+\frac{\zeta}{2})n< |X|\leq (\frac{2}{k+1}+\zeta)n$, then $X$ cannot have an $\alpha^2$-sparse cut.  

\begin{proof}
Suppose that $(X_1, X_2)$ is an $\alpha^2$-sparse cut of $G[X]$.  
By \ref{it:a} and \ref{it:b} we have $|X_i|> 2\alpha n$ for all $i\in [2]$. 
Indeed, if $|X_i|\leq 2\alpha n$, then \ref{it:b} implies that every vertex $v \in X_i$  
sends double edges to at least 
$$(1/(k+1)+\zeta -\alpha)|X|-2\alpha n \stackrel{(a)}{\geq} (1/(k+1)+\zeta -\alpha-2 (k+1)\alpha )|X|\geq |X|/(k+1)$$ vertices in $X_{3-i}$; 
but then $e^+_G(X_1,X_2)\geq |X_1||X_2|/(k+1) > \alpha ^2 |X_1||X_2|$, a contradiction as 
 $(X_1, X_2)$ is an $\alpha^2$-sparse cut of $G[X]$.

In fact, by combining $|X_i|>2\alpha n$ with \ref{it:c} we have 
\begin{align}\label{boundonXi}
   |X_i|>\left (\frac{1}{k+1}+\zeta-2\alpha \right )n \text{ for all } i\in [2].  
\end{align}
Indeed, if $|X_i|\leq (\frac{1}{k+1}+\zeta-2\alpha)n$, then by \ref{it:c}, there are $|X_i|-\alpha^2n>\alpha n$ vertices in $X_i$ that send double edges to at least $\alpha n$ vertices in $X_{3-i}$, contradicting the fact that $(X_1, X_2)$ is an $\alpha^2$-sparse cut.  

For all $i\in [2]$, let 
$$X_i':=\left \{v\in X_i: d_G(v,X_i)\leq |X_i|+\left (\frac{1}{k+1}+\zeta-9(k+1)\alpha \right)n \right \}.$$  
We now claim that for all  $i\in [2]$,
\begin{equation}\label{eq:X'}
 |X_i'|\leq \frac{\alpha}{6(k+1)}n.  
\end{equation}
Indeed, suppose that $|X_i'|>\frac{\alpha}{6(k+1)}n$ for some $i\in [2]$.  Then by \ref{it:c}, for all but at most $\alpha^2n$ vertices $v\in X_i'$, we have 
\begin{align*}
d_G(v, X_{3-i}) & \geq |X|+\left (\frac{1}{k+1}+\zeta-\alpha \right )n-\left (|X_i|+ \left (\frac{1}{k+1}+\zeta-9(k+1)\alpha \right)n\right ) \\ & \geq |X_{3-i}|+8(k+1)\alpha n,\end{align*}
which implies that
$$e^+_G(X_1, X_2)\geq (|X_i'|-\alpha^2 n)\cdot 8(k+1)\alpha n>\frac{\alpha}{8(k+1)}n\cdot 8(k+1)\alpha n = \alpha^2 n^2\geq \alpha^2|X_1||X_2|,$$
contradicting the fact that $(X_1, X_2)$ is an $\alpha^2$-sparse cut of $G[X]$.

Now set 
$$X_2'':=\left \{v\in X_1'\cup X_2': d_G(v,X_2\setminus (X_1'\cup X_2'))\geq \left (1+\frac{1}{k+1}+\zeta-\alpha \right )|X_2|-\frac{\alpha}{6(k+1)}n\right \},$$ 
and let $X_1'':=(X_1'\cup X_2')\setminus X_2''$.  
Note that by \ref{it:b}, \eqref{eq:X'} and the definition of $X_1''$, we have that for all $v\in X_1''$, 
$$d_G(v, X_1\setminus (X_1'\cup X_2'))\geq \left (1+\frac{1}{k+1}+\zeta-\alpha \right )|X_1|-\frac{\alpha}{6(k+1)}n.$$  
Set $V_1:=(X_1\setminus X_1')\cup X_1''$ and $V_2:=(X_2\setminus X_2')\cup X_2''$. Note that for all $i\in [2]$, we have 
\begin{equation}\label{eq:XiVi}
|X_i|-\frac{\alpha}{6(k+1)}n\leq |V_i|\leq |X_i|+\frac{\alpha}{6(k+1)}n.  
\end{equation}
Now  \eqref{boundonXi} and \eqref{eq:XiVi} imply that 
\begin{equation}\label{eq:Vi}
|V_i|>\left (\frac{1}{k+1}+\frac{\zeta}{2}\right )n, 
\end{equation}
and thus \ref{it:i} is satisfied.

If $v \in X_i \setminus X'_i $, then 
by definition of $X'_i$,
$d_G(v,X_i) \geq \left (1+\frac{1}{k+1}+\zeta-9(k+1)\alpha \right )|X_i|$ and so by  \eqref{eq:XiVi} we have that 
$d_G(v,V_i) \geq \left (1+\frac{1}{k+1}+\zeta-9(k+1) \alpha \right )|X_i| -\frac{\alpha }{6(k+1)}n$. Further, if  $v \in X''_i  $, then 
$d_G(v,V_i) \geq \left (1+\frac{1}{k+1}+\zeta- \alpha \right )|X_i| -\frac{\alpha}{ 6(k+1)}n$.
Thus, for all $i\in [2]$, 
\begin{align*}
\delta(G[V_i])& \geq \left (1+\frac{1}{k+1}+\zeta-9(k+1) \alpha \right )|X_i|-\frac{\alpha}{6(k+1)}n \\
&\stackrel{\eqref{eq:XiVi}}\geq \left (1+\frac{1}{k+1}+ \zeta-9(k+1)\alpha \right) \left (|V_i|-\frac{\alpha}{6(k+1)}n \right )-\frac{\alpha}{6(k+1)}n\\
&\geq \left (1+\frac{1}{k+1}+\zeta-9(k+1)\alpha \right )|V_i|-\frac{\alpha}{2(k+1)}n\\
&\stackrel{\eqref{eq:Vi}}{\geq} \left (1+\frac{1}{k+1}+\zeta-10(k+1)\alpha \right )|V_i|,
\end{align*}
so \ref{it:ii} is satisfied.  

By the definition of $X_i'$ and \eqref{eq:X'} we have that for all $i\in [2]$, all but at most $|X_1'|+|X_2'|\leq \frac{\alpha}{3(k+1)}n< \alpha n$ vertices $v\in V_i$ have 
\begin{align*}
d_G(v,V_i) & \stackrel{\eqref{eq:XiVi}}{\geq}
|X_i|+\left (\frac{1}{k+1}+\zeta-9(k+1)\alpha \right )n- \frac{\alpha}{6(k+1)}n \\ &
\stackrel{\eqref{eq:XiVi}}{\geq}
|V_i|+\left (\frac{1}{k+1}+\zeta-9(k+1)\alpha \right )n-\frac{\alpha}{3(k+1)}n \\ & >|V_i|+\left (\frac{1}{k+1}+\zeta-10(k+1)\alpha \right)n,
\end{align*}
and thus \ref{it:iii} is satisfied.  

Finally, we have that $(V_1, V_2)$ is an $\alpha$-sparse cut since 
\begin{align*}
e^+_G(V_1, V_2)&\leq  e^+_G(X_1, X_2)+|X_1''||X_2''|+|X_1||X_2''|+|X''_1||X_2| \\
& \leq \alpha^2|X_1||X_2|
+\frac{\alpha^2}{36(k+1)^2}n^2
+\frac{\alpha }{3(k+1)}n |X_1|+ 
\frac{\alpha }{3(k+1)}n |X_2|
\\
&\leq 2 \alpha^2 |V_1||V_2|+ \alpha^2 |V_1||V_2| +2\alpha |V_1||V_2|/5+2\alpha |V_1||V_2|/5
\\ & \leq \alpha |V_1||V_2|,
\end{align*}
where the penultimate inequality uses   \eqref{eq:XiVi} and \eqref{eq:Vi} in the form $n<(k+1)|V_i|$ and consequently $|X_i|<(1+\frac{\alpha}{6})|V_i|$ for all $i\in [2]$. 
\end{proof}

We are now ready to prove Proposition~\ref{prop:structure}.

\begin{proof}[Proof of Proposition \ref{prop:structure}]
Let $G$ be an $n$-vertex digraph where $n\geq n_0$ and  $\delta(G)\geq (1+\frac{1}{k+1}+\zeta)n$. The proof proceeds by running an iterative procedure to obtain the desired partition $\{V_1,\dots, V_t\}$ of $V(G)$.
Set $\alpha _b:= \alpha ^{2^{k-b}}$ for $0\leq b\leq k$. Note that $\nu \ll \tau \ll \alpha_b \ll \zeta \ll 1/k$ for any $0\leq b\leq k$; this hierarchy will allow us to apply both Lemma~\ref{lem:sparse} and Lemma~\ref{lem:clean} throughout the procedure (with $\alpha _b$ playing the role of $\alpha$). 

By Lemma~\ref{lem:sparse}, either $G$ is a robust $(\nu, \tau)$-outexpander or $G$ has an $\alpha_0$-sparse 
cut.  In the former case, conditions (i) and (ii) of Proposition~\ref{prop:structure} hold (with $t=1$), so we are done. 

Thus, suppose $G$ has an $\alpha_0$-sparse cut. By Lemma~\ref{lem:clean}, there exists an $\alpha_1$-sparse cut $(V^1 _1, V^1 _2)$ of $G$ satisfying conditions \ref{it:i}--\ref{it:iii} of the lemma. 
That is, for $i \in [2]$:
\begin{itemize}
    \item $|V^1_i|> (\frac{1}{k+1}+\frac{\zeta}{2})n$;
\item $\delta(G[V^1_i])\geq (1+\frac{1}{k+1}+\zeta-10(k+1)\alpha_1)|V^1_i|\geq (1+\frac{1}{k+1}+\zeta-\alpha_2)|V^1_i|$;
\item $d_G(v,V^1_i)\geq |V^1_i|+(\frac{1}{k+1}+\zeta-10(k+1)\alpha_1)n \geq |V^1_i|+(\frac{1}{k+1}+\zeta-\alpha_2)n$ for all but at most $\alpha _1 n=\alpha ^2 _2 n$ vertices $v\in V^1_i$.
\end{itemize}
If both $G[V^1 _1]$ and $G[V^1 _2]$ are robust $(\nu, \tau)$-outexpanders then we set $V_1:=V^1_1$ and $V _2:=V^1 _2$ and terminate the procedure.
Otherwise, we will continue the iterative procedure as follows.

At each subsequent step of the procedure we will begin with a partition $\{V^b_1, \dots, V^b_a \}$ of $V(G)$
such that, for each $ i \in [a]$: 
\begin{itemize}
    \item[(C1)] $|V^b_i|> (\frac{1}{k+1}+\frac{\zeta}{2})n$;
\item[(C2)] $\delta(G[V^b_i])|\geq (1+\frac{1}{k+1}+\zeta-\alpha_{b+1})|V^b_i|$;
\item[(C3)] $d_G(v,V^b_i) \geq |V^b_i|+(\frac{1}{k+1}+\zeta-\alpha_{b+1})n$ for all but at most $\alpha _{b} n=\alpha ^2 _{b+1} n$ vertices $v\in V^b_i$.
\end{itemize}
We then proceed as follows:
For each class $V^b _i$ such that $G[V^b_i]$ is a robust $(\nu, \tau)$-outexpander, we will perform no further partition of $V^b _i$; that is, $V^b _i$ will be one of the final classes in the partition $\{V_1,\dots, V_t\}$ of $V(G)$.
So in particular, if all classes $V^b _i$ have this property then we terminate the procedure.

Otherwise, for any other class  $V^b _i$ we will have that 
$G[V^b _i]$ has an $\alpha_{b+1} ^2$-sparse cut  by Lemma~\ref{lem:sparse}. 
We can then apply Lemma~\ref{lem:clean} to $X:=V^b _i$ to obtain  a partition $\{V^b_{i,1},V^b_{i,2} \}$ of $V^b _i$ such that 
$(V^b_{i,1},V^b_{i,2})$ is an $\alpha_{b+1} $-sparse cut in $G[V^b _i]$, and so that, 
for each $ j \in [2]$:
\begin{itemize}
    \item $|V^b_{i,j}|> (\frac{1}{k+1}+\frac{\zeta}{2})n$;
\item $\delta(G[V^b_{i,j}])\geq (1+\frac{1}{k+1}+\zeta-10(k+1)\alpha_{b+1})|V^b_{i,j}|\geq (1+\frac{1}{k+1}+\zeta-\alpha_{b+2})|V^b_{i,j}|$;
\item $d_G(v,V^b_{i,j})\geq |V^b_{i,j}|+(\frac{1}{k+1}+\zeta-10(k+1)\alpha_{b+1})n \geq |V^b_{i,j}|+(\frac{1}{k+1}+\zeta-\alpha_{b+2})n$ for all but at most $\alpha _{b+1} n=\alpha ^2 _{b+2} n$ vertices $v\in V^b_{i,j}$.
\end{itemize}

Once we have performed this step for all classes $V^b _i$ having a sparse cut, by relabeling we obtain a new partition 
$\{V^{b+1}_1, \dots, V^{b+1}_{a'} \}$ of $V(G)$ where $a+1\leq a' \leq 2a$.
In particular, we order the  classes so that the first class(es) in this new partition are subsets of $V^b _1$, the next class(es) are subsets of $V^b _2$, and so on. Moreover, if a class $V^b _i$ was partitioned in the last step, then the class $V^b_{i,1}$ must appear immediately before the class $V^b_{i,2}$ in the new partition  $\{V^{b+1}_1, \dots, V^{b+1}_{a'} \}$. Note that $\{V^{b+1}_1, \dots, V^{b+1}_{a'} \}$ satisfies (C1)--(C3) (with $b+1$ and $a'$ playing the roles of $b$ and $a$, respectively). We then continue to the next step of the procedure.

Crucially, this procedure must terminate for some $b \leq k-1$. Indeed, at every step of the process we either terminate the procedure (if all classes induce robust $(\nu, \tau)$-outexpanders), or create a new partition of $V(G)$ with at least one  more partition class than before. By (C1), we can have at most $k$ classes in $\{V^b_1, \dots, V^b_a \}$; this therefore means the procedure does terminate after at most $ k-1$ steps.
Let $\{V_1, \dots, V_t \}$ denote the partition at the end of this procedure. Then (C1) implies that Proposition~\ref{prop:structure}(i) holds.

Since the procedure only terminates when each class $V_i$ in the partition induces a
 robust $(\nu, \tau)$-outexpander, together with (C2) this implies that
Proposition~\ref{prop:structure}(ii) holds.

If $t \geq 2$, consider any $V_i$ and $V_j$ where $1\leq i <j \leq t$. By the way in which we construct and order our partitions at every step, it must be the case that there is some $b \in [k-1]$ and $V^b _{i'},  V^b _{j'} \subseteq V(G)$ such that 
$V_i \subseteq V^b _{i'}$; $V_j \subseteq V^b _{j'}$; $(V^b _{i'},  V^b _{j'})$ is an $\alpha _b$-sparse cut in $G[V^b _{i'}\cup V^b _{j'}]$.
As $|V_i|, |V_j|\geq \frac{n}{k+1}$ this implies that
$$
e^+_G (V_i,V_j) \leq \alpha  _b \cdot |V^b _{i'}|  |V^b _{j'}| \leq 
\alpha  _{k-1} (k+1)^2|V_i| |V_j| \leq \alpha  |V_i| |V_j|,
$$
which implies that Proposition~\ref{prop:structure}(iv) holds.

Finally, if $t \geq 2$, suppose for a contradiction that $e^-_G(V_i, V_j)\leq \frac{n^2}{(k+1)^2}$ for some $1\leq i<j\leq t$.    If $|V_i|\leq |V_j|$, using Proposition~\ref{prop:structure}(iv) we obtain that
$$
|V_i|\left (1+\frac{1}{k+1}+\zeta \right )n-2|V_i|^2\leq e_G(V_i, V(G)\setminus V_i)\leq |V_i|(n-|V_i|-|V_j|)+\frac{n^2}{(k+1)^2}+\alpha  n^2,
$$ 
which implies that
$$\left (1+\frac{1}{k+1}+\zeta \right )n\leq n-|V_j|+|V_i|+\frac{n^2}{(k+1)^2|V_i|}+\frac{\alpha n^2}{|V_i|}\stackrel{(i)}{<}\left (1+\frac{1}{k+1}+\zeta \right )n,$$ a contradiction.  Similarly, if $|V_j|\leq |V_i|$, we analogously obtain a contradiction by considering $e_G(V_j, V(G)\setminus V_j)$. Thus, Proposition~\ref{prop:structure}(iii) holds.
\end{proof}

\section{Embedding the Hamilton cycle}\label{sec:embed}

Our main result in this section is the following.  

\begin{proposition}\label{prop:main}
Let $0<1/n_0 \ll \rho\ll\nu\ll\tau\ll\eta \ll 1$ and let $n\geq n_0$.  If $G$ is an $n$-vertex digraph and $\{V_1, \dots, V_t\}$ is a partition of $V(G)$ such that:
\begin{enumerate}
\item for all $i\in [t]$, $|V_i|\geq \eta n$ and consequently $t\leq \frac{1}{\eta}$;
\item for all $i\in [t]$, $G[V_i]$ is a robust $(\nu,\tau)$-outexpander with $\delta^0(G[V_i])\geq \eta|V_i|$;
\item if $t \geq 2$, then for all $1\leq i<j\leq t$, $e^+_G(V_i, V_j)\geq \frac{n}{\rho^2}$, 
\end{enumerate}
then $G$ contains every orientation of a Hamilton cycle, except possibly the directed Hamilton cycle. 
\end{proposition}

Note that Theorem~\ref{thm:main} follows easily from  
Propositions~\ref{prop:structure} and~\ref{prop:main}.
\begin{proof}[Proof of Theorem~\ref{thm:main}]
Let $\eta >0$ and define $k \in \mathbb N$  such that
$1/(k+1) \ll \eta $. Set $\eta ':=1/(k+1)$ and define additional constants so that
$$
0<1/n_0 \ll \rho \ll \nu \ll \tau \ll \alpha \ll \zeta \ll \eta'.
$$
Let $G$ be a digraph on $n\geq n_0$ vertices with $\delta (G) \geq (1+\eta )n\geq (1+\eta '+\zeta )n$.
By applying Proposition~\ref{prop:structure}, and reversing the ordering of the partition of $V(G)$ outputted, we obtain a partition $\{V_1, \dots, V_t\}$ of $V(G)$ so that:
\begin{itemize}
    \item  for all $i\in [t]$, $|V_i|\geq (\eta' +\frac{\zeta}{2})n$;
\item for all $i\in [t]$, $G[V_i]$ is a robust $(\nu,\tau)$-outexpander with $\delta(G[V_i])\geq (1+\eta '+\frac{\zeta}{2})|V_i|$, and thus $\delta ^0(G[V_i])\geq \eta '|V_i| $;
\item if $t\geq 2$, then for all $1\leq i<j\leq t$, $e^+_G(V_i, V_j)> \frac{n^2}{(k+1)^2}$.
\end{itemize}
Applying Proposition~\ref{prop:main} (with $\eta '$ playing the role of $\eta $) implies that  $G$ contains every orientation of a Hamilton cycle,  except for perhaps the directed Hamilton cycle.
In the case when $G$ is  strongly connected, $G$ contains a directed Hamilton cycle by Theorem~\ref{ghthm}.
\end{proof}

Before proceeding with the proof of Proposition~\ref{prop:main} we need some further notation. 
A \emph{segment} of an oriented cycle $C$ (resp.~oriented path $P$) is an oriented path that forms an induced subgraph of $C$ (resp.~$P$). A \emph{directed segment} of an oriented cycle or path is a segment that induces a directed path.
Note that we will often write a segment in the form $x_1\dots x_{\ell}$ to indicate the ordered sequence of vertices this segment contains; this notation does not tell us anything about how the edges in the segment are oriented.
Given a segment $P$ of an oriented cycle or path $C$,
a non-endpoint vertex $v$ of $P$ is a \emph{switch} if it is either a source or a sink in $P$.

The proof of Proposition~\ref{prop:main}  splits into two cases depending on whether the orientation of a Hamilton cycle we wish to embed contains a directed segment of order at least $\beta n$ or not.  The following observation explicitly describes the two cases.

\begin{observation}\label{obs:cases}
Let $0<\beta\leq 1$ and let $C$ be an oriented cycle on $n\geq 3$ vertices.  Either
\begin{enumerate}
\item $C$ contains a directed segment on $\floor{\beta n}$ vertices, or
\item every segment of $C$ on  $\floor{\beta n}$ vertices contains a switch. \qed
\end{enumerate}
\end{observation}

The next observation will allow us to connect up oriented paths that lie in different vertex classes in the partition of $V(G)$ in Proposition~\ref{prop:main}.

\begin{observation}\label{obs:bip}
Let $\rho >0$, let $G$ be an $n$-vertex digraph and let $X, Y\subseteq V(G)$ be disjoint with  $e^+_G(X,Y)\geq \frac{n}{\rho}$.  If $X':=\{x\in X: d^+_G(x, Y)\geq \frac{1}{\rho}\}$ and  $Y':=\{y\in Y: d^-_G(y, X)\geq \frac{1}{\rho}\}$, then $|X'|, |Y'|\geq \frac{1}{\rho}$.
\end{observation}

\begin{proof}
If say $|X'|<\frac{1}{\rho}$, then we have $e^+_G(X,Y)< |X'||Y| +(|X|-|X'|)\frac{1}{\rho}<\frac{1}{\rho}|Y| +\frac{1}{\rho}|X| \leq \frac{n}{\rho}$, a contradiction. 
\end{proof}

We also need a result that allows us to span each robust outexpander $G[V_i]$ with a collection of oriented paths (which will then be connected together to form our desired orientation of a Hamilton cycle in $G$).  This will be derived from a strengthening of the following theorem of Taylor \cite[Theorem~49]{Tay}.  

\begin{theorem}[Taylor~\cite{Tay}]\label{thm:ham}
Let $0<1/n_0\ll \nu \ll \tau\ll \eta<1$ and let $G$ be a digraph on $n\geq n_0$ vertices.  If $G$ is a robust $(\nu, \tau)$-outexpander with $\delta^0(G)\geq \eta n$, then $G$ contains every orientation of a Hamilton cycle.
\end{theorem}

One can modify the proof of Theorem \ref{thm:ham} to obtain the following. 

\begin{theorem}[Universally $k$-linked]\label{thm:link}
Let $0<1/n_0\ll \beta \ll \nu \ll \tau\ll \eta$, let $k\in \mathbb N$  with $k\leq 1/\beta^2$,
 let $G$ be a digraph on $n\geq n_0$ vertices, and let $Q_1, \dots, Q_k$ be a collection of oriented paths with $|Q_i|=:\ell_i\geq 1/\beta$ for all $i\in [k]$ such that $\ell_1+\dots+\ell_k=n$.  If $G$ is a robust $(\nu, \tau)$-outexpander with $\delta^0(G)\geq \eta n$, then for all distinct vertices $u_1, \dots, u_{k}, v_1, \dots, v_k\in V(G)$, there exists a collection of disjoint oriented paths $P_1, \dots, P_k$ in $G$ such that for all $i\in [k]$, $P_i$ starts with $u_i$ and ends with $v_i$, and $P_i$ is a copy of $Q_i$. 
\end{theorem}

An explanation of how to deduce Theorem~\ref{thm:link} from Taylor's work~\cite{Tay} is given in the appendix.

Finally, we need the following simple observation which is just an easy special case of Havet and Thomass\'e's theorem~\cite{HavT}.

\begin{observation}\label{obs:ros}
Let $T$ be a transitive tournament.  For all oriented paths $Q$ with $|Q|\leq |T|$, there is a copy of $Q$ in $T$. \qed
\end{observation}

We now proceed to the  proof of Proposition~\ref{prop:main}.   

\begin{proof}[Proof of Proposition~\ref{prop:main}]
We first choose a constant $\beta$ so that $\rho\ll \beta\ll \nu$. Let $G$ be an $n$-vertex digraph as in the proposition. 
Note we may assume that $t\geq 2$, otherwise the proposition follows immediately from Theorem~\ref{thm:ham}.

Let $C$ be an $n$-vertex oriented cycle that is  not a directed cycle. We will define an embedding $f:V(C) \rightarrow V(G)$ of $C$ into $G$. We will split into cases depending on how $C$ is oriented. In each case, 
we will first select the edges used in $f$ that go between different vertex classes $V_i, V_j$ for each $1\leq i<j\leq t$ (and possibly a small number of other edges); we then apply Theorem~\ref{thm:link} to find collections of oriented paths in each $G[V_i]$ that, together with the previously selected edges, complete the embedding $f$.

\noindent
\tbf{Case 1: There exists a directed segment of $C$ on $\floor{\beta n}$ vertices.}  Let $\ell\geq \floor{\beta n}$ denote the number of vertices in the longest directed segment of $C$.  For each $j \in [t]$, let $\ell_j:=\sum _{i=1} ^{j} |V_i|$ and set $\ell_0:=0$. Note that $\ell_t=n$.

{\it Subcase 1(a): $0\leq n-\ell \leq \eta n/2$.}
Write $C=x_1\dots x_nx_1$ where $x_1$ is a source in $C$, $x_{\ell}$ is a sink in $C$ and $P:=x_1\dots x_{\ell}$ is a longest directed segment of $C$.
In particular, $x_1x_2, x_1x_n \in E(C)$.
Let $P':=x_{\ell+1}\dots x_{n}$ be the segment of $C$ from $x_{\ell+1}$ to $x_n$; note that $P'$ is empty if $\ell=n$.

For each $j \in [t]$, define $P_j:=x_{\ell_{j-1}+1}\dots x_{\ell_{j}}$ to be the segment of $C$ that starts at $x_{\ell_{j-1}+1}$ and ends at $x_{\ell_{j}}$. 
Note that for $j \in [t-1]$, $P_j$ is contained in $P$ and so is a directed segment of $C$, whereas  $P_t$ contains $P'$ (so may not be a \emph{directed} segment).
Note that the edges in $E(C)\setminus E(P_1\cup \dots \cup P_t)$ are precisely $x_1x_n$ and $x_{\ell_j}x_{\ell_j+1}\in E(C)$ for each $j\in [t-1]$.

We will define $f$ so that $P_j$ is embedded into $G[V_j]$ for each $j \in [t]$.
First, we choose $f(x_1) \in V_1$; $f(x_n) \in V_t$; $f(x_{\ell_j} )\in V_j$ and $f(x_{\ell_j+1} )\in V_{j+1}$ 
for each $j \in [t-1]$ so that:
\begin{itemize}
\item all selected vertices in this step are distinct;
\item $f(x_1)f(x_n) \in E(G)$;
\item $f(x_{\ell_j} )f(x_{\ell_j+1} ) \in E(G)$ for each $j \in [t-1]$.
\end{itemize}
Note that we can select such vertices in $G$ by Observation~\ref{obs:bip}.

Next, for each $j \in [t]$, we find a copy $P'_j$ of $P_j$  in $G[V_j]$ 
that starts at $f(x_{\ell_{j-1}+1})$ and ends at $f(x_{\ell_{j}})$. Indeed, we obtain these oriented paths by applying Theorem~\ref{thm:link} to each of $G[V_1], \dots, G[V_{t}]$. 

By concatenating the  paths $P'_1,\dots, P'_t$ and the edges we selected between vertex classes, we obtain our desired embedding $f$ of $C$ in $G$.

\smallskip

{\it Subcase 1(b): $\eta n/2 < n-\ell \leq n-\lfloor \beta n \rfloor$.}
Write $C=x_1\dots x_nx_1$ where $x_1$ is a source in $C$, $x_{n-\ell+2}$ is a sink in $C$ and $P:=x_1 x_n\dots x_{n-\ell+2}$ is a longest directed segment of $C$.
Let $P':=x_2 \dots x_{n-\ell+1}$ be the segment of $C$ from $x_2$ to $x_{n-\ell+1}$.
Note that $|P'|=n-\ell>\eta n/2$.

Set $D:=\floor{\rho n}$ and $q:=\ceil{\frac{n-\ell}{D}}<1/\rho$ and partition $P'$ into segments $P_1':=x_2\dots x_{n_1}$; $P_2':=x_{n_1+1} \dots x_{n_2}$; $\dots$; $P'_{q}:=x_{n_{q-1}+1}\dots x_{n-\ell+1}$ that are chosen to be as equal sized as possible. In particular, 
 for all $i\in [q]$, certainly $D/2 \leq |P_i'| \leq D$.  

Each segment $P'_i$ will be embedded fully into one of our vertex classes $V_j$. On the other hand, $P$ will be embedded across all of the vertex classes $V_1,\dots, V_t$ (see Figure~\ref{fig:case1b}). Before we can embed these segments, we need to construct the edges of $C$ that will go between the vertex classes.

To help with this, we define the  auxiliary oriented path $Q:=w_1w_2\dots w_q$ such that for all $i\in [q-1]$, $w_i w_{i+1}\in E(Q)$ if $x_{n_i} x_{n_i+1}\in E(P')$, and $w_{i+1} w_{i}\in E(Q)$ if $x_{n_i+1} x_{n_i}\in E(P')$.

Let $U_1, \dots, U_t$ be disjoint sets such that for all $j\in [t]$, $|U_j|=\floor{\frac{|V_j|-\rho n}{D}}$ and the elements of $U_j$ are named $u^j_1, \dots, u^j_{|U_j|}$.  Consider an auxiliary transitive tournament $T$ on vertex set $U_1\cup\dots\cup U_t$ such that $u^j_{k} u^{j'}_{k'}\in E(T)$ if and only if $j<j'$ or $j=j'$ and $k<k'$.  Note that 
$$|T|=\sum_{j\in [t]} \bigfloor{\frac{|V_j|-\rho n}{D}}\geq \bigceil{\frac{(1-\beta/2)n}{D}}\geq \bigceil{\frac{n-\ell}{D}}=|Q|.$$  
By Observation \ref{obs:ros}, there is an embedding $\phi$ of $Q$ into $T$.  

We now use this auxiliary embedding $\phi$ as a `blueprint' to determine in which of the classes $V_1, \dots, V_t$ each segment $P'_i$ will be embedded into:
If $\phi(w_i)\in U_j$, then we will eventually embed $P'_i$ into $V_j$.  Furthermore if $\phi(w_i)\in U_j$ and $\phi(w_{i+1})\in U_{j'}$ and $w_i w_{i+1}\in E(Q)$, then by the way we defined $T$ it must be the case that $j\leq j'$.  If $j<j'$, this ensures that we can use Observation~\ref{obs:bip} to find an edge in $G$ from $V_j$ to $V_{j'}$ that will connect $f(x_{n_i})$ (the last vertex of $P_i'$) to $f(x_{n_i+1})$ (the first vertex of $P_{i+1}'$); if $j=j'$ we  use that $\delta^0(G[V_j])\geq \eta |V_j|$  to find an edge in  $G[V_j]$ that will connect $f(x_{n_i})$ to $f(x_{n_i+1})$.
One can  argue similarly 
if $w_{i+1}w_{i}\in E(Q)$ (with the roles of $j$ and $j'$ switched).

\begin{figure}[ht]
\centering
 \includegraphics[scale=1.3]{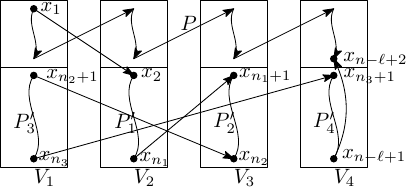} 
\caption{Case 1(b): An example with $q=4$}\label{fig:case1b}
\end{figure}

We now know in which classes we wish to embed each segment $P'_i$ of $P'$. This tells us how many vertices 
are `left' in each class $V_j$ that we need to cover using $P$. 
In particular, note that by the way we defined $T$, there are at least $\rho n$ vertices in each class $V_j$ that will not be used for embedding $P'$;\footnote{Specifically, at most $\floor{\frac{|V_j|-\rho n}{D}}$ of the segments $P'_i$ (each of which contain at most $D$ vertices) will be embedded into a class $V_j$.} so these vertices will have to be covered by $P$.
We will embed $P$ so that it goes through each class $V_1,\dots, V_t$ sequentially. That is, $x_1$ will be embedded in $V_1$ and we continue to use vertices in $V_1$ until the correct number of vertices are covered, then we jump over to $V_2$ (using a single edge between $V_1$ and $V_2$) and repeat this process;  finally $x _{n-\ell+2}$ will be embedded into $V_t$.
Before we can embed the whole of $P$, for each $j \in [t-1]$, we need to choose the single edge in $G$ that goes between $V_j$ and $V_{j+1}$ that will be used in the embedding of $P$. We can use Observation~\ref{obs:bip} to select such  edges so that they are all disjoint, and disjoint from all the edges selected previously that will be used to connect up the segments $P'_i$.

Also note that if $\phi(w_1)\in U_{j_1}$, then regardless of the value of $j_1$, we  can use Observation~\ref{obs:bip} (or the fact that $\delta^0(G[V_1])\geq \eta |V_1|$ if $j_1=1$) to find an edge from $V_{1}$ to $V_{j_1}$ that will connect $f(x_{1})$ to $f(x_{2})$. Likewise, if $\phi(w_q)\in U_{j_q}$, then regardless of the value of $j_q$ we can  find an edge from $V_{j_q}$ to $V_t$ which will connect $f(x_{n-\ell+1})$ to $f(x_{n-\ell+2})$. 

In summary, we now have selected all the edges in $G$ required to connect between the segments $P'_1,\dots, P'_q, P$, and all edges that will go between different classes in the embedding of $P$ into $G$.  
One can now apply Theorem~\ref{thm:link} to each induced subgraph $G[V_i]$ to obtain the embeddings of the segments $P'_1,\dots, P'_q$ as well as the remaining edges needed for $P$. This therefore completes the embedding $f$ of $C$ into $G$.

\smallskip

\noindent
\tbf{Case 2: Every segment on $\floor{\beta n}$ vertices contains a switch.}  Write $C=x_1x_2\dots x_nx_1$ where $x_1$ is a source in $C$.  In this case we have the following crucial property:
\begin{equation}\label{eq:switchy}
\text{Every segment of $C$ on at least $2\beta n$ vertices contains both a source and a sink.} 
\end{equation}

Our plan is to partition $C$ into $t$  segments $P_1, \dots, P_t$ and construct an embedding $f:V(C)\to V(G)$ of $C$ in $G$ so that for all $i\in [t-1]$, $f(P_i)\subseteq V_i\cup V_{i+1}$ and $f(P_t)\subseteq V_t$, and such that for all $i\in [t]$, $f(P_i)$ starts and ends in $V_i$.  We now explain the steps we take to achieve this embedding.

\smallskip

\tbf{Step 1: Partitioning $C$ into $t$  segments.}  Set $n_0:=0$ and $n_t:=n$.  Let $P_1:=x_{n_0+1}\dots x_{n_1}$ be the minimal segment in $C$ that starts at $x_{n_0+1}=x_1$ and where $x_{n_1}x_{n_1+1}\in E(C)$ and $|P_1|\geq |V_1|$;
note that as we are in Case~2, $P_1$ exists  and 
$|P_1|\leq |V_1|+\beta n$. Since $|V_2|\geq \eta n \gg \beta n$, we have that $|P_1|$ is much smaller than $|V_1|+|V_2|$.

Let $ s\in [t-2]$ and suppose we have chosen segments $P_1=x_{n_0+1}\dots x_{n_1}, P_2=x_{n_1+1}\dots x_{n_2},$ $\dots, P_{s}=x_{n_{s-1}+1}\dots x_{n_s}$ such that for all $i\in [s]$, $P_i$ is the minimal segment of $C$ that starts at $x_{n_{i-1}+1}$
such  that
 $x_{n_i}x_{n_i+1}\in E(C)$ and $|V_1|+\dots+|V_i|\leq |P_1|+\dots+|P_i|\leq |V_1|+\dots+|V_i|+\beta n$.
We then choose $P_{s+1}:=x_{n_s+1}\dots x_{n_{s+1}}$ minimally such that $x_{n_{s+1}}x_{n_{s+1}+1}\in E(C)$ and  $|P_1|+\dots+|P_{s+1}|\geq |V_1|+\dots+|V_{s+1}|$; again this is possible as we are in Case~2, and in fact we have $|P_1|+\dots+|P_{s+1}|\leq |V_1|+\dots+|V_{s+1}|+\beta n$.  Finally, we let $P_t:=x_{n_{t-1}+1}\dots x_{n_t}$ and note that by the way the other segments $P_i$ were chosen, we have $|V_t|-\beta n\leq |P_t|\leq |V_t|$.

\smallskip

\tbf{Step 2: Setting up the connections and fixing the imbalance.}  
Ideally we would like to have had that $|P_i|=|V_i|$ for all $i \in [t]$. Then we could embed $P_1$ fully into $V_1$ and then jump over to $V_2$ and embed $P_2$ there, and so forth, to obtain an embedding of $C$ into $G$.
However, we  have that $|P_i| \leq |V_i|+\beta n$, and in particular there may be a gap between 
$\sum_{i=1}^s|P_i|$ and $\sum_{i=1}^s|V_i|$ for some of the 
$s\in [t-1]$.

As such, in this step we will embed constantly-many  segments of $C$ into $G$ in such a way that along $C$ there is at least a large constant gap between consecutive pairs of such segments.  The goal is to make all of the required connections between $V_i$ and $V_{i+1}$ for each $i \in [t-1]$, while at the same time correcting the gap between $\sum_{i=1}^s|P_i|$ and $\sum_{i=1}^s|V_i|$ for each $s\in [t-1]$.

We begin by using Observation~\ref{obs:bip} to obtain a matching $M$ in $G$ consisting of one edge $v_1 v_n=v_{n_0+1} v_{n_t}$ from $V_1$ to $V_t$, and for all $i\in [t-1]$, one edge $v_{n_i} v_{n_i+1}$ from $V_i$ to $V_{i+1}$.  For each $i\in [t]$, set $f(x_{n_i}):=v_{n_i}$, and for each $i\in [t-1]\cup \{0\}$, set
$f(x_{n_i+1}):=v_{n_i+1}$.

For all $s\in [t-1]$, set $d_s:=\sum_{j=1}^s|P_j|-\sum_{j=1}^{s}|V_j|\geq 0$.  
Let $s\in [t-1]$.  We first informally describe our plan depending on the value of $d_s$.  If $d_s=0$, then by the definition of $P_s$, we have that $x_{n_s}x_{n_s+1}\in E(C)$ and we have already embedded the edge $x_{n_s}x_{n_s+1}$ from $V_s$ to $V_{s+1}$.  As there is no imbalance, no further work is necessary in this case.  If $d_s>0$, then first note that the minimality of $P_s$ combined with the fact that we are in  Case~2 implies that $x_{n_s}$ is source in $C$ and $d_s\leq \beta n$.   Let $a_s+1<n_s$ be a maximal index such that $x_{a_s+1}$ is a sink in $C$.  Note that $P^*_s:=x_{n_s}x_{n_s-1}\dots x_{n_s-d_s}\dots x_{a_s+1}$ is a directed segment on at most $\beta n$ vertices (and on at least $d_s+1$ vertices).  Now there are two subcases: Either $0<d_s\leq \frac{\eta}{6\beta}$ and we are able to fix the imbalance by embedding $d_s$ sinks from $P_s$ in $V_{s+1}$ (in a way that will be made precise shortly), or $\frac{\eta}{6\beta}<d_s\leq \beta n$ and we use $\floor{\frac{\eta}{12\beta}}$  sinks as before, together with the segment $P^*_s=x_{n_s}x_{n_s-1}\dots x_{a_s+1}$ to fix the imbalance (roughly by embedding $x_{a_s+d_s-\floor{\frac{\eta}{12\beta}}}\dots x_{a_s+1}$ in $V_{s+1}$, $x_{a_s}$ in $V_s$, and $x_{n_s}\dots x_{a_s+d_s-\floor{\frac{\eta}{12\beta}}+1}$ in $V_s$). We now make this all precise.

\begin{figure}[ht]
   \centering
    \begin{subfigure}[t]{0.5\textwidth}
       \centering
         \includegraphics[scale=1]{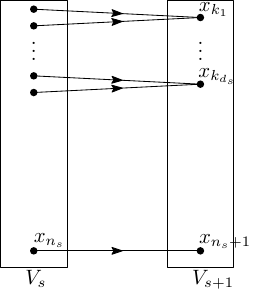} 
\caption{$0<d_s\leq \frac{\eta}{6\beta}$}\label{fig:case2i}
    \end{subfigure}%
    ~ 
    \begin{subfigure}[t]{0.5\textwidth}
       \centering
         \includegraphics[scale=1]{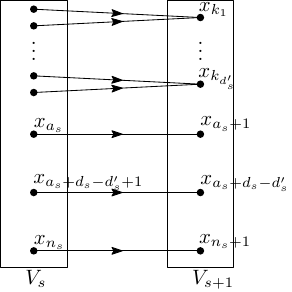} 
\caption{$\frac{\eta}{6\beta}<d_s\leq \beta n$}\label{fig:case2ii}
    \end{subfigure}
    \caption{Case 2, Step 2}\label{fig:case2}
\end{figure}

Let $s\in [t-1]$.  If $d_s=0$, we do nothing further, so suppose $d_s>0$.  
First suppose that $0< d_s\leq \frac{\eta}{6\beta}$.
Since $|V_s|\geq \eta n$, how we defined $n_s$ and $n_{s-1}$, together with \eqref{eq:switchy}, implies that $n_s -n_{s-1} \geq \eta n -2\beta n$.
Using this with \eqref{eq:switchy}  ensures that we can select $d_s$ sinks $x_{k^s_1}, \dots, x_{k^s_{d_s}}$ on $P_s$ 
where 
\begin{itemize}
    \item[(D1)] $n_{s-1}+2\beta n< k^s_1<k^s_2<\dots <k^s_{d_s}< n_s-2\beta n$;
    \item[(D2)] $x_{k^s_1}, \dots, x_{k^s_{d_s}}$ do not lie on the directed segment $P^*_s$;
    \item[(D3)] for all $1\leq i<j\leq d_s$ we have that $k^s_j-k^s_i\geq \beta n$.
\end{itemize}
In fact, notice (D2) follows immediately from (D1) and the definition of $P^*_s$.

   We now use Observation~\ref{obs:bip} to find $d_s$ disjoint subgraphs $\Lambda^s_1, \dots , \Lambda^s_{d_s}$ of $G$ where, for each $i \in [d_s]$,
   $V(\Lambda^s_i):=\{v_{k^s_i-1}, v_{k^s_i+1}, v_{k^s_i}\}$ with $v_{k^s_i-1}, v_{k^s_i+1}\in V_s$ and $v_{k^s_i}\in V_{s+1}$ such that $E(\Lambda^s_i)=\{v_{k^s_i-1} v_{k^s_i}, v_{k^s_i+1} v_{k^s_i} \}$ (see Figure~\ref{fig:case2i}). 
   Further, each $\Lambda^s_i$ is disjoint from the matching $M$ we previously selected in $G$ and disjoint from any other $\Lambda^{r}_i$'s we may have chosen previously (for $r<s$).
   Note we can ensure that these conditions hold since $e^+_G(V_s,V_{s+1}) \geq n/\rho ^2$ which is comfortably clear of the corresponding condition in Observation~\ref{obs:bip}.
   Now for all $i\in [d_s]$, set $f(x_{k^s_i}):=v_{k^s_i}$, $f(x_{k^s_i-1}):=v_{k^s_i-1}$, and $f(x_{k^s_i+1}):=v_{k^s_i+1}$.  
   This ensures that we will embed precisely $d_s$ vertices of $P_s$ into $V_{s+1}$.

Next suppose that $\frac{\eta}{6\beta}<d_s\leq \beta n$.  Set $d_s':=\floor{\frac{\eta}{12\beta}}$. Analogously to before, select $d_s'$ sinks $x_{k^s_1}, \dots, x_{k^s_{d_s'}}$ on $P_s$ satisfying (D1)--(D3).
 We then use Observation~\ref{obs:bip} to find $d_s'$
 disjoint subgraphs $\Lambda^s_1, \dots, \Lambda^s_{d'_{s}}$ of $G$ where, for each $i \in [d'_s]$,
   $V(\Lambda^s_i)=:\{v_{k^s_i-1}, v_{k^s_i+1}, v_{k^s_i}\}$ with $v_{k^s_i-1}, v_{k^s_i+1}\in V_s$ and $v_{k^s_i}\in V_{s+1}$ such that $E(\Lambda^s_i)=\{v_{k^s_i-1} v_{k^s_i}, v_{k^s_i+1} v_{k^s_i} \}$. 
   Further, each $\Lambda^s_i$ is disjoint from the matching $M$ we previously selected in $G$ and disjoint from any other $\Lambda^{r}_i$'s we may have chosen previously (for $r<s$).
   Now for all $i\in [d'_s]$, set $f(x_{k^s_i}):=v_{k^s_i}$, $f(x_{k^s_i-1}):=v_{k^s_i-1}$, and $f(x_{k^s_i+1}):=v_{k^s_i+1}$.

Next we use Observation~\ref{obs:bip} to find disjoint edges $v_{a_s} v_{a_s+1}$ and $v_{a_s+d_s-d_s'+1} v_{a_s+d_s-d_s'}$ from $V_s$ to $V_{s+1}$ in $G$ (that are disjoint from $M$ and all the $\Lambda^s_i$).  Set $f(x_{a_s}):=v_{a_s}$, $f(x_{a_s+1}):=v_{a_s+1}$, $f(x_{a_s+d_s-d_s'+1}):=v_{a_s+d_s-d_s'+1}$, and $f(x_{a_s+d_s-d_s'}):=v_{a_s+d_s-d_s'}$ (see Figure~\ref{fig:case2ii}).

Thus, when we finish the embedding $f$ of $C$ into $G$ in Step~3,  precisely 
$d_s$ vertices from $P_s$ will be embedded into $V_{s+1}$: $d'_s$ vertices playing the roles of the sinks $x_{k^s_1}, \dots, x_{k^s_{d_s'}}$ and  the $d_s-d'_s$ vertices that are in the segment 
$x_{a_s+d_s-d_s'}\dots x_{a_s+1}$ of $C$ (which lies fully in $P^*_s$). Note too that, since $|P^*_s|\geq d_s$, this segment $x_{a_s+d_s-d_s'}\dots x_{a_s+1}$ is of distance at least $d_s'=\floor{\frac{\eta}{12\beta}}$ from $x_{n_s}$ in $C$. This property is vital in ensuring we can apply Theorem~\ref{thm:link} in Step~3.

\smallskip

\tbf{Step 3: Finishing the embedding.}  At this point we have embedded only a constant number of  segments of $C$ with the key property that between any two already embedded segments of $C$, there is at least a large constant gap (at least $d_s'=\floor{\frac{\eta}{12\beta}}$) between the ends of the segments.  In particular, in the case when $d_s=0$ this is ensured as the
only segment of $C$ we have
 embedded using vertices both from $V_s$ and $V_{s+1}$ is $x_{n_s}x_{n_s+1}$ (which is a segment of $C$ far from the segments of $C$ corresponding to $d_{s-1}$ and $d_{s+1}$). In the case when $0< d_s\leq \frac{\eta}{6\beta}$ this is ensured by (D1) and (D3). In the case when $\frac{\eta}{6\beta}<d_s\leq \beta n$ this is ensured by (D1), (D3) and by the choice of the segment $x_{a_s+d_s-d_s'}\dots x_{a_s+1}$ of $P^*_s$ that we have chosen to be embedded into $V_{s+1}$.  

For each $i \in [t]$, let $V'_i $ be the subset of $V_i$ obtained by deleting all middle (sink) vertices in the subgraphs $\Lambda^s_j$ of $G$ obtained in Step~2. Note that
$|V_i\setminus V'_i|\leq \frac{\eta}{6\beta}$; thus, each $G[V'_i]$ is a robust $(\nu/2, 2\tau)$-outexpander with $\delta ^0 (G[V'_i]) \geq \eta |V'_i|/2$.
Therefore, to finish the embedding $f$, we apply Theorem~\ref{thm:link} to each $G[V'_i]$ to embed the remaining segments of $C$ into the appropriate places.
In particular, note that the number of remaining segments required in each $G[V'_i]$
is  certainly  at most $(\eta/ 12 \beta )^2$, so we can indeed apply Theorem~\ref{thm:link} here.
\end{proof}

\section{Concluding remarks}\label{sec:conc}
In this paper we have asymptotically determined the minimum degree threshold for forcing an arbitrary orientation of a Hamilton cycle in a digraph (Theorem~\ref{thm:main}). It would be interesting to obtain an exact version of this result. 

\begin{problem}\label{proby}
If $G$ is a sufficiently large $n$-vertex digraph  with $\delta(G)\geq n+1$, then does $G$ contain every orientation of a Hamilton cycle (except for the directed Hamilton cycle when $G$ is not strongly connected)?
\end{problem}
Recall that one cannot lower the minimum degree condition in Problem~\ref{proby} for anti-directed Hamilton cycles~\cite{Cai}.
However, it would  be interesting to know if a minimum degree of $\delta(G)\geq n$ is already enough to force every orientation of a Hamilton cycle, other than the directed and anti-directed ones.

\smallskip

The next example, a variant of that of Cai~\cite{Cai} (see Figure \ref{fig:cai1}), yields a digraph of high minimum degree that has no anti-directed Hamilton path.

\begin{proposition}\label{newestprop}
For all $m\geq 2$ there exists a digraph $F_{2m+1}$ on $n:=2m+1$ vertices with $\delta^0(F_{2m+1})=\frac{n-1}{2}$ (and thus $\delta(F_{2m+1})=n-1$) such that $F_{2m+1}$ has no anti-directed Hamilton path.  
\end{proposition}

\begin{proof}
Define $F_{2m+1}$ so that $V(F_{2m+1})$ is the disjoint union of $X_1, X_2$ and $\{y_1,y_2\}$ where $|X_1|=m$ and $|X_2|=m-1$, and where the edge set is defined as follows (see Figure \ref{fig:cai2}): there are all possible edges in both directions between $X_1$ and $X_2$; all possible edges from $y_1$ to $X_1$; 
all possible edges from $X_1$ to $y_2$;
all possible edges from $y_2$ to $X_2$; all
 possible edges from $X_2$ to $y_1$; there is the edge from $y_2$ to $y_1$. 

Note that the only anti-directed path between $y_1$ and $y_2$ is the single edge $y_2y_1$.  So if there is an anti-directed Hamilton path, it must be of the form $v_1\dots v_iy_1y_2v_{i+3}\dots v_n$.  This forces $y_1$ to be a sink and $y_2$ to be a source on the path which implies that $v_i, v_{i+3}\in X_2$.  Thus at least half of the vertices from $v_1\dots v_i$ are from $X_2$ and at least half of the vertices from $v_{i+3}\dots v_n$ are from $X_2$, contradicting the fact that $|X_2|<|X_1|$.  
\end{proof}

\begin{figure}[htbp]
\centering
\begin{subfigure}[t]{.48\textwidth}
\centering

\includegraphics[scale=.7]{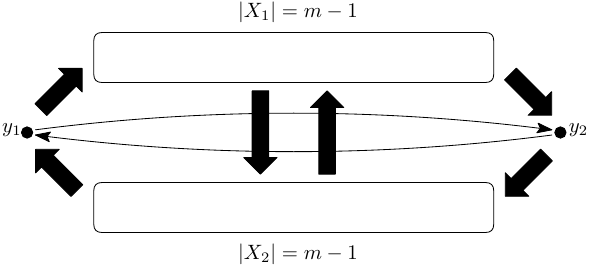}
\caption{Cai's example of an $n$-vertex digraph $F^1_{2m}$ with minimum semi-degree $m=\frac{n}{2}$ with no anti-directed Hamilton cycle.}
\label{fig:cai1}
\end{subfigure}\hfill
\begin{subfigure}[t]{.48\textwidth}
\centering

\includegraphics[scale=.7]{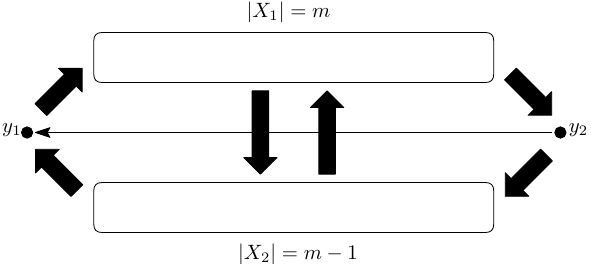}
\caption{An example of an $n$-vertex digraph $F_{2m+1}$ with minimum semi-degree $m=\frac{n-1}{2}$ with no anti-directed Hamilton path.}
\label{fig:cai2}
\end{subfigure}
\caption{Counterexamples for anti-directed Hamilton cycles and paths.}
\label{fig:counterexamples}
\end{figure}

As mentioned in the introduction, a positive answer to the following problem would provide a common generalization of Havet and Thomass\'e's tournament theorem~\cite{HavT} and Corollary~\ref{ghcor}.

\begin{problem} Is the following true?  For all $n\geq 2$, every $n$-vertex digraph $G$ with $\delta(G)\geq n-1$ contains every orientation of a Hamilton path unless the path is anti-directed and $G$ is isomorphic to the digraph $F_{2m+1}$ in Proposition~\ref{newestprop}, or one of Gr\"unbaum's three exceptional tournaments on 3,5,7 vertices respectively.  In particular, if $n$ is even, then every $n$-vertex digraph $G$ with $\delta(G)\geq n-1$ contains every orientation of a Hamilton path.
\end{problem}

From the perspective of obtaining results that unify the minimum degree setting and the setting of tournaments, it may also be interesting to find a characterization of all those (strongly connected)
 $n$-vertex digraphs with $\delta(G) \geq n-1$ that do not contain every orientation of a Hamilton cycle.

\smallskip

Finally, we note that Proposition~\ref{prop:structure} provided crucial structure that we exploited in the proof of Theorem~\ref{thm:main}. This tool is likely to have several further applications; we will explore this in future work.

\smallskip

\section*{Acknowledgment}
The authors are  grateful to the referees for their careful reviews.

{\noindent \bf Open access statement.}
	This research was funded in part by  EPSRC grants  EP/V002279/1 and EP/V048287/1. For the purpose of open access, a CC BY public copyright licence is applied to any Author Accepted Manuscript arising from this submission.

{\noindent \bf Data availability statement.}
There are no additional data beyond that contained within the main manuscript.

\bibliographystyle{abbrv}
\bibliography{references.bib}


\section{Appendix: Proofs of Theorems~\ref{cor2} and~\ref{thm:link}}\label{sec:appendix}
\subsection{Deriving Theorem~\ref{thm:link}}
In this subsection we prove  Theorem~\ref{thm:link}. Before this, we introduce the main results used to obtain Theorem~\ref{thm:link}.  Very recently, the first author~\cite{D26} gave alternate regularity-free proofs of Lemma~\ref{lem:splitRobust}, Theorem~\ref{thm:hamCon}, and Theorem~\ref{thm:link}.  Thus together with the regularity-free proof of Lemma~\ref{prop:shortConnect} already given in~\cite{Tay}, Theorem~\ref{thm:main} and Theorem~\ref{cor2} can now be proved without the use of the regularity lemma.

The following lemma of Taylor~\cite{Tay} allows us to find any not too short, not too long oriented path between any pair of vertices in a robust outexpander.\footnote{Formally one applies Proposition~48 and Lemma~58 from~\cite{Tay} to obtain Lemma~\ref{prop:shortConnect}.} 

\begin{lemma}[Taylor~\cite{Tay}]\label{prop:shortConnect}
Let $0<1/n_0 \ll \nu \leq \tau \ll \eta \ll 1$, let $n \geq n_0$, and let $\lceil 2/\nu ^3 \rceil \leq k \leq \nu^3 n/4$. If $G$ is an $n$-vertex digraph with $\delta^0(G) \geq \eta n$ such that $G$ is a robust $(\nu,\tau)$-outexpander, then for all distinct  $x,y \in V(G)$ and all oriented paths $P$ of length $k$, there is a copy of $P$ in $G$ that starts at $x$ and ends at $y$.
\end{lemma}
Lemma~\ref{prop:shortConnect} will allow us to  find all the required short oriented paths in the proof of Theorem~\ref{thm:link}.
The next two results will enable us to construct the longer oriented paths. The first of these results 
generalizes~\cite[Lemma 60]{Tay} which was also proved in a different form in~\cite[Lemma B.5]{Granet}. 

\begin{lemma}[Splitting robust expanders]\label{lem:splitRobust}
Let $0<{1}/{n_0}\ll  \nu' \ll \ep  \ll \nu \ll \tau \ll \tau ' \ll \eta$
and let $n\geq n_0$ be an integer. Let $m_0, m_1, \dots, m_t\in \mathbb N\cup \{0\}$  such that for all $i\in [t]$, $m_i\geq \ep^{1/3} n$;  $m_0\leq \eps ^{1/4} n$; $m_0 +m_1+\dots+m_t=n$. Suppose that $G$ is an $n$-vertex digraph with $\delta^0(G)\geq \eta n$ and $G$ is a robust $(\nu, \tau)$-outexpander. If $W_0 \subseteq V(G)$ is any set of $m_0$ vertices, then there is a partition $\{W_0, W_1, \dots, W_t\}$ of $V(G)$ such that for all $i\in [t]$:
\begin{itemize}
    \item   $|W_i|=m_i$ and $G[W_i]$ is a robust $(\nu', \tau')$-outexpander;
    \item  $d^+_G(x,W_i), d^-_G(x,W_i) \geq \eta m_i/4$ for every  $x \in V(G)$.
\end{itemize}
\end{lemma}
Lemma~60 in~\cite{Tay} yields the $W_0=\emptyset$, $t=2$ case of the lemma. In fact, the argument there can easily be tweaked to prove Lemma~\ref{lem:splitRobust}. As such, we only provide a detailed proof sketch. Note
that taking a random partition of $V(G) \setminus W_0$ to obtain
$\{W_0, W_1, \dots, W_t\}$ needs care since one does not wish to consider too
 many events so as to be unable to apply a union bound; a subtle use of the regularity lemma allows us to restrict the number of events considered.  In~\cite{D26}, this issue is overcome by using the container method.  

\begin{proof}[Sketch proof of Lemma~\ref{lem:splitRobust}]
    Define an additional constant $d$ so that 
    $$0<{1}/{n_0}\ll  \nu' \ll \ep \ll d \ll \nu \ll \tau \ll \tau ' \ll \eta.$$
    Let $G$ be a digraph on $n \geq n_0$ vertices as in the statement of the lemma. Let $G':=G \setminus W_0$ and set $n':=|G'|$.
    As $|W_0|\leq \eps^{1/4}n$, $G'$ is a robust $(\nu /2, 2\tau)$-outexpander with $\delta^0(G')\geq \eta n'/2$.

    Apply the regularity lemma for digraphs to $G'$ with parameters  $\eps, d$ (see, e.g.,~\cite[Lemma~39]{Tay} for the version of the regularity lemma that we use). We therefore obtain a partition of $V(G')$ into clusters $V_1,\dots, V_k$ of the same size $m$ and an exceptional set $V_0$, with $|V_0|\leq \eps n'$. Crucially, for every distinct $i,j \in [k]$, $(V_i,V_j)$ forms an $\eps$-regular pair in $G'$ of density either $0$ or at least $d$.

We now randomly partition $V(G')$ into classes $W_1, \dots, W_t$ where 
$|W_i|=m_i$ for all $i \in [t]$. By Chernoff's bound for the hypergeometric distribution, with high probability,  for each $i \in [t]$ and $x \in V(G)$ we have 
$d^+_G(x,W_i), d^-_G(x,W_i) \geq \eta m_i/4$.
    Moreover, with high probability, for each $i \in [t]$ and
    $j \in [k]$, we have 
    $$|W_i \cap V_j|\geq (1-\eps) m \cdot \frac{m_i}{n} \geq \eps ^{1/2} m.$$
    For each $i \in [t]$ and
    $j \in [k]\cup \{0\}$, set $V^i _j := W_i \cap V_j$.
    Thus, $V^i_0, V^i _1, \dots, V^i _k$ is a partition of $W_i$. Moreover, as $|V^i_j|\geq \eps ^{1/2} m$ for each $j \in [k]$,
    we have the following property:
    \begin{itemize}
        \item[(P1)] If $(V_{j_1},V_{j_2})$ forms an $\eps$-regular pair of density at least $d$ in $G'$, then $(V^i_{j_1},V^i_{j_2})$ forms an 
        $\eps^{1/2}$-regular pair of density at least $d-\eps$ in $G'[W_i]$.
    \end{itemize}
    In particular, (P1) implies that the reduced digraph $R_i$ of $G'[W_i]$ is the same as the reduced digraph $R$ of $G'$. As $G'$ is a robust $(\nu /2, 2\tau)$-outexpander with $\delta^0(G')\geq \eta n'/2$, \cite[Lemma~50]{Tay} implies that $R$ is a 
    robust $(\nu /4, 4\tau)$-outexpander with $\delta^0(R)\geq \eta k/4$.

    Finally, as argued at the end of the proof of Lemma~60 in~\cite{Tay},  since the reduced digraph $R_i=R$ of $G'[W_i]$ is a  robust $(\nu /4, 4\tau)$-outexpander, this implies that 
    $G'[W_i]=G[W_i]$ is a robust $(\nu ', \tau ')$-outexpander. 
    Indeed, if one considers any set $S \subseteq W_i$ where 
    $\tau ' m_i \leq |S| \leq (1-\tau ')m_i$, $S$ must intersect 
    many $V^i _1, \dots, V^i _k$ significantly (certainly more 
    than $\tau k$ of these classes). Let $Q$ be the set of such 
    significantly intersected $V^i _j$. 
    In particular, we have 
    that the robust outneighborhood $RN^+ _{\nu /4} (Q)$ of $Q$ in 
    $R_i$ satisfies  $|RN^+ _{\nu /4} (Q)| \geq |Q| + \nu k /4$. 
    Moreover, most vertices in each class $V^i_j \in RN^+ _{\nu /4} (Q)$ lie in the $\nu '$-robust out-neighborhood $RN^+ _{\nu '} (S)$ of $S$ in $G[W_i]$. 
    A simple calculation now implies that $|RN^+ _{\nu '} (S)|\geq |S| +\nu ' m_i$, as desired.
\end{proof}

We also require the following generalization of Theorem~\ref{thm:ham}.
\begin{theorem}[Universally Hamilton connected]\label{thm:hamCon}
Let $0<{1}/{n_0}\ll \nu \ll \tau\ll \eta$ and let $G$ be a digraph on $n\geq n_0$ vertices.  If $\delta^0(G)\geq \eta n$ and $G$ is a robust $(\nu, \tau)$-outexpander, then for all distinct $x, y\in V(G)$ and any oriented path $P$ on $n$ vertices, there exists a copy of $P$ in $G$ that starts at $x$ and ends at $y$.
\end{theorem}
This result is easily derived from the following version of Theorem~\ref{thm:ham}.
\begin{theorem}\label{thm:hamCon2}
Let $0<{1}/{n_0}\ll \nu \ll \tau\ll \eta$ and let $G$ be a digraph on $n\geq n_0$ vertices.  Suppose that $\delta^0(G)\geq \eta n$ and $G$ is a robust $(\nu, \tau)$-outexpander. Let $C$ be any oriented cycle  on $n$ vertices and fix $x_C \in V(C)$. Given any $v_G \in V(G)$, there
is a copy of $C$ in $G$ in which $x_C$ is embedded onto $v_G$.\qed
\end{theorem}
Note that the proof of Theorem~\ref{thm:ham} in~\cite{Tay} immediately yields Theorem~\ref{thm:hamCon2}. Indeed, in this proof, when embedding an $n$-vertex oriented cycle $C$, segments $Q_i$ of the cycle  are embedded greedily (see Step~4 in the proof of Theorem~\ref{thm:ham} in~\cite[Theorem~49]{Tay}).
Thus, we may line things up so that one of the $Q_i$ contains $x_C$ (roughly in the middle of $Q_i$). We then (via Lemma~\ref{prop:shortConnect}) greedily construct a copy of $Q_i$ in $G$ that embeds $x_C$ onto $v_G$.

We now explain how Theorem~\ref{thm:hamCon2} implies Theorem~\ref{thm:hamCon}.

\begin{proof}[Proof of Theorem~\ref{thm:hamCon}]
Let $G$ and $P$ be as in the statement of the theorem.
Write $P=v_1\dots v_n$. We will define an auxiliary digraph $G^*$ whose precise construction depends on the orientations of the edges incident to $v_1$ and $v_n$ in $P$.
 For instance, suppose that $|N^+_P(v_1)|=|N^+_P(v_n)|=1$ (and $|N^-_P(v_1)|=|N^-_P(v_n)|=0$); the other cases can be handled analogously.
Construct the digraph $G^*$ from $G$ by deleting $x$ and $y$ and adding a vertex $v_{xy}$  so that (i) $N_{G^*} ^-(v_{xy}):=N^+_G(x)\setminus \{y\}$ and (ii) $N_{G^*} ^+(v_{xy}):=N^+_G(y)\setminus \{x\}$.
Note that $G^*$ is a robust $(\nu /2, 2\tau)$-outexpander with $\delta ^0(G^*)\geq \eta |G'|/2$.

Let $C$ be the oriented cycle obtained from $P$ by deleting its startpoint $v_1$ and its endpoint $v_n$, and adding a new vertex $v_{1,n}$ that receives an edge from $v_2$ and sends out an edge to $v_{n-1}$. By Theorem~\ref{thm:hamCon2} there is a copy of $C$ in $G^*$ in which $v_{1,n}$ is embedded onto $v_{xy}$. The definition of $N_{G^*} ^+(v_{xy})$ and $N_{G^*} ^-(v_{xy})$ now ensures that this copy of $C$ corresponds to a copy of $P$ in $G$ that starts at $x$ and ends at $y$, as desired.
\end{proof}


With the above results at hand, we can now easily prove Theorem~\ref{thm:link}.

\begin{proof}[Proof of Theorem \ref{thm:link}]
Define additional constants so that
$$0
<{1}/{n_0}\ll \nu' \ll \ep \ll \beta  \ll \nu \ll \tau \ll \tau ' \ll \eta.
$$
Let $n \geq n_0$ and
let $G$, $Q_1,\dots, Q_k$, $u_1,\dots,u_k,v_1,\dots, v_k$ be as defined in the statement of the theorem.

Since $\beta  \ll \nu$, certainly $|Q_i| \geq 1/\beta \geq  \ceil{2/\nu^3}$. In particular, we can repeatedly apply Lemma~\ref{prop:shortConnect} to obtain the required copy $P_i$ of $Q_i$ for each path $Q_i$ with $|Q_i| \leq \eps^{1/4} n$ having the property that for all such $i$, $V(P_i)\cap \{u_1, ..., u_k,v_1, ..., v_k\}=\{u_i,v_i\}$. (Formally, whenever constructing such a $P_i$, one should first delete all the other oriented paths $P_j$ constructed before, as well as all of $u_1,\dots,u_k,v_1,\dots, v_k$ other than $u_i$ and $v_i$, and then apply Lemma~\ref{prop:shortConnect} to the resulting subgraph of $G$. This ensures all the oriented paths $P_j$ are disjoint. All together, this is at most $k\eps^{1/4} n +2k\leq \eps^{1/5}n$ forbidden vertices.)

For each oriented path $P_i$ already constructed, delete its vertices from $G$; let $G'$ be the resulting induced subgraph of $G$.
Note that at most $\frac{1}{\beta^2}\cdot \eps^{1/4} n \leq \eps ^{1/5} n$ vertices of $G$ have been deleted. Thus, $G'$ is 
a robust $(\nu/2, 2 \tau)$-outexpander with $\delta ^0(G)\geq \eta |G'|/2$.
By relabeling if necessary, we may assume that we still need to construct the copies of $Q_1,\dots, Q_t$ in $G$ for some $t \leq 1/\beta ^2$, and now $|Q_i|\geq \eps ^{1/4} n$ for all $i \in [t]$.

Set $W_0:= \{u_1,\dots, u_t, v_1,\dots, v_t\}$;
so certainly, $m_0:=|W_0| \leq \eps ^{1/4} |G'|$.
For each $i \in [t]$, define $m_i:=|Q_i|-2 \geq \eps^{1/4} n-2 \geq \eps ^{1/3}|G'|$. Thus, we can apply
Lemma~\ref{lem:splitRobust}
to obtain a partition $\{W_0, W_1, \dots, W_t\}$ of $V(G')$ such that for all $i\in [t]$:
\begin{itemize}
    \item   $|W_i|=m_i$ and $G'[W_i]$ is a robust $(\nu', \tau')$-outexpander;
    \item  $d^+_{G'}(x,W_i), d^-_{G'}(x,W_i) \geq \eta m_i/8$ for every  $x \in V(G')$.
\end{itemize}

Next add $u_i,v_i$ to $W_i$ for each $i \in [t]$. We now have that $G'[W_i]$ is a robust $(\nu'/2, 2\tau')$-outexpander with 
$\delta^0 (G'[W_i]) \geq \eta |W_i|/10$
(for each $i \in [t]$).
Finally, Theorem \ref{thm:hamCon} implies that in each $G'[W_i]$ we have the desired copy $P_i$ of $Q_i$ that starts at $u_i$ and ends at $v_i$. This completes the proof.
\end{proof}

\subsection{Proof of Theorem~\ref{cor2}}
In this subsection we prove Theorem~\ref{cor2}. For this, we will use the following strengthening of Lemma~\ref{prop:shortConnect}.
\begin{lemma}\label{prop:shortConnectgen}
Let $0<1/n_0 \ll \nu \ll \tau \ll \eta \ll 1$. Let $n \geq n_0$ and let $\lceil 2/\nu ^3 \rceil < k \leq n$.  If $G$ is an $n$-vertex digraph with $\delta^0(G) \geq \eta n$ such that  $G$ is a robust $(\nu,\tau)$-outexpander, then for all distinct $x,y \in V(G)$ and all oriented paths $P$ on $k$ vertices, there is a copy of $P$ in $G$ that starts at $x$ and ends at $y$.
\end{lemma}
\begin{proof}
Define additional constants so that
 $$0<{1}/{n_0}\ll \nu' \ll \ep \ll \nu \ll \tau \ll \tau ' \ll \eta.$$ 
Let $n\geq n_0$ and define $G$, $P$, $x$, $y$ as in the statement of the lemma.

If $\lceil 2/\nu ^3 \rceil < |P| \leq \nu^3 n/4$, then we are immediately done by Lemma~\ref{prop:shortConnect}.
So suppose that $|P| > \nu^3 n/4$.

If $\nu ^3 n/4 <|P|< (1-\eps^{1/3})n$, set 
 $W_0:=\{x,y\}$; $m_0:=|W_0|=2$; $m_1:=|P|-2 \geq \eps ^{1/3}n$; $m_2:=n-|P| \geq \eps^{1/3} n$. 
 Now we apply Lemma~\ref{lem:splitRobust} to obtain  sets $W_1, W_2 \subseteq V(G)$. Add $x$ and $y$ to $W_1$; we have that $G[W_1]$ is a robust $(\nu '/2, 2\tau')$-outexpander with $\delta ^0(G[W_1]) \geq \eta |W_1|/5$. Note that $|W_1|=|P|$. Theorem~\ref{thm:hamCon} now implies that $G[W_1]$ contains a copy of $P$ that starts at $x$ and ends at $y$.

Finally, if
 $|P|\geq  (1-\eps^{1/3})n$, consider any induced subgraph $G'$ of $G$ on $|P|$ vertices that contains $x$ and $y$. Then $G'$ is a robust $(\nu /2, 2\tau)$-outexpander with $\delta ^0(G') \geq \eta |G'|/2$. Theorem~\ref{thm:hamCon}  implies that $G'$ contains a copy of $P$ that starts at $x$ and ends at $y$, as desired.
\end{proof}

We will apply the following special case of a well-known result of Bondy and Simonovits~\cite{BS74}.

\begin{theorem}[Bondy and Simonovits~\cite{BS74}]\label{thm:BS} If $G$ is a graph on $n$ vertices with $e(G)>200n^{3/2}$, then $G$ contains a cycle of length $2\ell$ for all $2\leq \ell\leq 2n^{1/2}$.
\end{theorem}

In the proof of Theorem~\ref{cor2} we also apply a result of Aldred, Holten and Min \cite[Theorem 2.3]{aldred} which is a slight strengthening of a well-known result of Bondy~\cite{bondy}.

\begin{theorem}[Aldred, Holten and Min~\cite{aldred}]\label{thm:AHM}
Let $G$ be a graph on $n$ vertices such that for all distinct $u,v\in V(G)$ with $uv\not\in E(G)$, $d_G(u)+d_G(v)\geq n-1$.  Either $G$ is pancyclic, or $G$ is isomorphic to one of the following:
\begin{itemize}
    \item[(i)] the union of two complete graphs $K_{(n+1)/2}$ that share a single vertex,
    \item[(ii)] a subgraph of $\overline{K}_{(n+1)/2}+K_{(n-1)/2}$,
    \item[(iii)] $K_{n/2, n/2}$, or
    \item[(iv)] $C_5$.
\end{itemize}
\end{theorem}

We are now ready to prove Theorem~\ref{cor2}.
\begin{proof}[Proof of Theorem~\ref{cor2}]
Define additional constants so that 
$$0<1/n_0 \ll \nu\ll\tau\ll\alpha\ll \zeta \ll \gamma, 1/k.$$
Let $n\geq n_0$ and let $G$ be an $n$-vertex digraph as in the statement of the theorem. Let $C$ be an oriented cycle on at most $n$ vertices that is not a directed cycle of length more than $\ceil {n/k}$. We will prove that $G$ contains $C$. We split into three cases.

\smallskip

\noindent \tbf{Case 1: $|C|\leq 1/ \nu ^4$.} 
Consider the (undirected) graph $G^*$ on vertex set $V(G)$ where $x$ and $y$ are adjacent in $G^*$ precisely if both $xy$ and $yx$ are edges in $G$.
Note that since $\delta (G^*) \geq \gamma n $, we may apply Theorem \ref{thm:BS} to $G^*$ to obtain every even length cycle of order up to $1/\nu^4$.
Hence, $G$ contains every oriented cycle of even length up to $1/\nu^4$.

We are therefore done unless $|C|$ is odd. However, in this case our argument above implies that $G$ contains a copy $C'$ of the `double edge' cycle on $|C|-1$ vertices. Let $x$ and $y$ be two neighboring vertices along $C'$. As $\delta (G) \geq (1+\gamma )n$, either
$d^+_G(x)+d^-_G(y) \geq (1+\gamma )n $ or 
$d^-_G(x)+d^+_G(y) \geq (1+\gamma )n $. Without loss of generality we may assume the former holds.
So there is some $z \in N^+_G (x) \cap N^- _G(y) \cap (V(G)\setminus V(C'))$.
Since every oriented cycle of odd order contains a vertex of in- and outdegree $1$, adding $z$ to $C'$ immediately yields a copy of $C$, as desired.

\smallskip

\noindent \tbf{Case 2: $ 1/ \nu ^4< |C| \leq \ceil {n/k}$.} 
In this case we apply Proposition~\ref{prop:structure} to obtain a partition $\{V_1,\dots, V_t\}$ of $V(G)$ into at most $k$ classes. In particular, for one of these classes $V_i$ we have that:
\begin{itemize}
    \item $|V_i| \geq \ceil {n/k}$;
    \item $G[V_i]$ is a robust $(\nu,\tau)$-outexpander with $\delta(G[V_i])\geq (1+\frac{1}{k+1}+\frac{\zeta}{2})|V_i|$.
\end{itemize}
Lemma~\ref{prop:shortConnectgen} implies that $G[V_i]$ contains a copy of $C$, as desired.
\smallskip

\noindent \tbf{Case 3: $|C|> \ceil {n/k}$ and $C$ is not a directed cycle.} 
Randomly select a subset $X$ of $V(G)$ of size $|C|$. Chernoff's bound for the hypergeometric distribution implies that, with high probability,
$\delta (G[X]) \geq (1+\gamma /2)|X|$. Theorem~\ref{thm:main} implies that $G[X]$ contains a copy of $C$.

\smallskip

We now prove the moreover part of the theorem.
Given $n \geq 2$, consider an $n$-vertex digraph $G$ with $\delta(G)\geq \floor{\frac{3n}{2}}-1$.  As before, let $G^*$ be the (undirected) graph on $V(G)$ where $x$ and $y$ are adjacent in $G^*$ precisely if both $xy$ and $yx$ are edges in $G$.  If $G^*$ is pancyclic, then clearly $G$ contains every oriented cycle of every possible length; so suppose not.

Note that $\delta (G^*) \geq \floor{\frac{n}{2}}$ and thus the hypothesis of Theorem~\ref{thm:AHM} is satisfied.  Since $G^*$ is not pancyclic we have that $G^*$ is isomorphic to a graph as described in Theorem~\ref{thm:AHM}  (i)--(iv).  Note that if we are in case (ii) where $G^*$ is isomorphic to a subgraph of $\overline{K}_{(n+1)/2}+K_{(n-1)/2}$, then since $\delta(G^*)\geq \floor{\frac{n}{2}}$, we have that $G^*$ contains $K_{(n+1)/2, (n-1)/2}$ as a subgraph.  

If (ii) or (iii) hold, then let $A$ be an independent set in $G^*$ with $|A|\geq n/2$.  Since $\delta(G)\geq \floor{\frac{3n}{2}}-1$, there is an edge from $G$ inside $A$. This edge, together with the structure obtained from $G^*$, ensures that $G$ contains every oriented cycle of every possible length.
Similarly, if (i) or (iv) hold then there is an edge present in $G$ that is not a double edge. One can then easily use this edge with the structure from $G^*$ to show that $G$ contains every oriented cycle of every possible length.
\end{proof}

\end{document}